\documentclass[11pt,a4paper]{amsart}

\usepackage{amsmath,amssymb,amsthm, amsfonts,enumerate,mathtools,pinlabel,float,makecell}
\usepackage[mathscr]{eucal}
\usepackage{amsbsy}
\usepackage{mathtools}
\usepackage{pinlabel}
\usepackage{graphicx}
\usepackage{multirow}
\usepackage{tabularx}
\usepackage{lscape}
\usepackage[myheadings]{fullpage}
\usepackage{rotating}
\allowdisplaybreaks
\usepackage{tikz-cd,tikz,pgf}
\usepackage{tikz-cd}
\usetikzlibrary{arrows.meta, bending}
\tikzset{C/.style={circle, minimum size=8mm,
                   node contents={},
                   append after command={\pgfextra{%
        \draw[-{Straight Barb[flex']}](\tikzlastnode.150) arc (150:450:4mm);}
                }}
        }
\usepackage{caption}
\usepackage{subcaption}
\usepackage{cite}
\usepackage{array}
\usepackage{xcolor}
\usepackage{calc}
\usepackage{comment}
\usepackage{soul}
\usepackage{hyperref}
\hypersetup{
    colorlinks=true,   
    linkcolor=blue,
    citecolor=blue,
}

\newtheorem{cor*}{Corollary}
\newtheorem{thm*}{Theorem}
\newtheorem{lem*}{Lemma}
\newtheorem{prop*}{Proposition}
\newtheorem{qstn*}{Question}
\newtheorem{theorem}{Theorem}[section]
\newtheorem{cor}{Corollary}[theorem]
\newtheorem{prop}[theorem]{Proposition}
\newtheorem{lem}[theorem]{Lemma}

\theoremstyle{definition}
\newtheorem{defn}[theorem]{Definition}
\newtheorem{exmp}[theorem]{Example}
\newtheorem{rem}[theorem]{Remark}

\newcommand{\G}{\mathcal{G}}
\newcommand{\Z}{\mathbb{Z}}
\newcommand{\N}{\mathbb{N}}
\newcommand{\Orb}{\mathcal{O}}

\newcommand{\D}{\mathscr{D}}
\newcommand{\E}{E_{n,m}^i}
\newcommand{\wlp}{(\mathscr{D}_a,(D_{\overline{G}},\Pi_{\overline{G}}))}

\DeclareMathOperator{\aut}{\mathrm{Aut}}
\DeclareMathOperator{\out}{\mathrm{Out}}
\DeclareMathOperator{\sig}{\mathrm{sig}}
\DeclareMathOperator{\lcm}{lcm}
\DeclareMathOperator{\map}{\mathrm{Mod}}

\DeclareMathOperator{\homeo}{\mathrm{Homeo}^{+}}

\everymath{\displaystyle}

\makeatletter
\@namedef{subjclassname@2022}{\textup{2022} Mathematics Subject Classification}
\makeatother

\begin{document}

\title[Liftability of periodic mapping classes under alternating covers]{Liftability of periodic mapping classes\\ under alternating covers}

\author{Rajesh Dey}
\address{(R. Dey) Department of Mathematics\\
Indian Institute of Science Education and Research Bhopal\\
Bhopal Bypass Road, Bhauri \\
Bhopal 462 066, Madhya Pradesh\\
India}
\email{rdeymath@gmail.com}

\author{Kashyap Rajeevsarathy}
\address{(K. Rajeevsarathy) Department of Mathematics\\
Indian Institute of Science Education and Research Bhopal\\
Bhopal Bypass Road, Bhauri \\
Bhopal 462 066, Madhya Pradesh\\
India}
\email{kashyap@iiserb.ac.in}
\urladdr{https://home.iiserb.ac.in/$_{\widetilde{\phantom{n}}}$kashyap/}

\author{Apeksha Sanghi}
\address{(A. Sanghi) School of Mathematical Sciences\\
National Institute of Science Education and Research Bhubaneswar\\
Jatni, Padanpur\\
Odisha 752050, India}
\email{apekshasanghi@niser.ac.in}



\begin{abstract} 
Let $S_g$ be the closed orientable surface of genus $g \geq 2$, and let $\mathrm{Mod}(S_g)$ be the mapping class group of $S_g$.  Let $A_n$ denote the alternating group on $n$ letters. We derive necessary and sufficient conditions under which a periodic mapping class has a conjugate that lifts under the branched cover $S_g \to S_g/A_n$ induced by an action of $A_n$ on $S_g$.  This provides a classification of the subgroups of $\mathrm{Mod}(S_g)$ that are isomorphic to $A_n \rtimes \Z_m$, up to a certain equivalence that we call weak conjugacy. As an application,  we show that for $n \geq 7$, such a subgroup of $\mathrm{Mod}(S_g)$ cannot have an irreducible periodic mapping class. Furthermore, we show that for $n \geq 5$ and $n \neq 6$, if the order of such a subgroup is greater than $5g-5$, then $m \leq 26$. Moreover,  for $g \geq 2$ and $n \geq 5$,  we establish that there exists no subgroup of $\mathrm{Mod}(S_g)$ that is isomorphic to $A_n \rtimes \mathbb{Z}$, where the $\mathbb{Z}$-component is generated by a power of a Dehn twist. Finally, we provide a complete classification of the weak conjugacy classes of such subgroups in $\mathrm{Mod}(S_{10})$ and $\mathrm{Mod}(S_{11})$.
\end{abstract}

\maketitle

\section{Introduction} 
\label{sec:intro}
 
Let $\map(S_g)$ denote the mapping class group of the closed orientable surface $S_g$ of genus $g \geq 2$. By the Nielsen realization theorem~\cite{SK}, a finite subgroup $H<\map(S_g)$ acts on $S_g$ via orientation-preserving isometries, inducing a finite-sheeted regular branched cover $p:S_g \to \mathcal{O}_H$, where $\mathcal{O}_H:=S_g/H$ is the quotient space of the $H$-action on $S_g$. The \textit{liftable mapping class group}, denoted by $\mathrm{LMod}(\mathcal{O}_H)$, is the subgroup of mapping classes in $\map(\mathcal{O}_H)$ represented by homeomorphisms that lift under the cover $p$. The problem of liftability of mapping classes originated from the pioneering works of Birman-Hilden in the early 1970s that started with~\cite{birman1971} and concluded with \cite{birman} (see also~\cite{margalit-winarski}). Since then, though the liftable mapping class groups under finite abelian covers have been widely studied \cite{broughton_normalizer,ghaswala_superelliptic,ghaswala_sphere,pankaj}, such an analysis is yet to be undertaken for covers induced by non-abelian finite simple groups. Since the most natural example of such a group is the alternating group $A_n$, our paper addresses the following question: 
\begin{qstn*} \label{qstn1} Given a conjugacy class of periodic mapping classes and an equivalence class of alternating subgroups of $\map(S_g)$, can one derive necessary and sufficient conditions under which there exists a representative mapping class that lifts under a representative alternating cover?
\end{qstn*}

As a first step towards answering this question, we introduce an equivalence relation, which we call \textit{weak conjugacy} (see Section \ref{sec:prelims}), and classify the weak conjugacy classes of subgroups of $\map(S_g)$ that are isomorphic to $A_n \rtimes \mathbb{Z}_m$ (see Proposition~\ref{prop:wlp}). As part of this classification, we provide a combinatorial tuple of integers $\D_E$ called a \textit{data set} that would encode the weak conjugacy class of subgroup $E = H \rtimes C$, where $H \cong A_n$ and $C \cong \mathbb{Z}_m$ (see Definition~\ref{defn:H_data_set}). Further, using the classical theory of group actions on surfaces~\cite{breuer}, we recover the combinatorial data $\D_H$ encoding the alternating component $H$ and data $D_{\bar{C}}$ encoding the cyclic group $\bar{C} \cong \mathbb{Z}_m$ that lifts to the cyclic component of $E$ under the following short exact sequence: 
$$1 \to H  \to E \to \bar{C} \to 1$$
(see Proposition~\ref{prop:wls}). Thus, by applying our combinatorial classification of the weak conjugacy classes, we obtain our main result, which provides an answer to the question posed above (see Theorems~\ref{thm:main1} and~\ref{thm:main2}). 
\begin{thm*}
\label{thm:main}
For $n \geq 5$ and $n \neq 6$, let $H < \map(S_g)$ that is isomorphic to an $A_n$. Then a periodic $\bar{G} \in \map(\mathcal{O}_H)$ lifts under the cover $S_g \to \mathcal{O}_H$ to a $G \in \map(S_g)$ if and only if there exists a data set $\D_E$ that has components $\D_H$ and $D_{\bar{C}}$, where $\bar{C} \cong \langle \bar{G} \rangle$.
\end{thm*}

We provide several applications of Theorem~\ref{thm:main} in Section~\ref{sec:apps}. As a first application, in Subsection~\ref{subsec:liftinv}, we provide a complete characterization of the liftability of involutions under alternating covers. In this direction, we derive sufficient conditions under which a group generated by an involution will lift under an $A_n$-cover to a group isomorphic to the symmetric group $\Sigma_n = A_n \rtimes \Z_2$ or the group $A_n \times \Z_2$ (see Theorem~\ref{thm:alternating_extensions}). As another application relating to involutions, we show the following:
\begin{cor*}
Let $E < \map(S_g)$ such that $E \cong A_n \rtimes \Z_m$, where $n \geq 10$, $m$ is odd if $E$ is a direct product, and $m/2$ is odd when $E$ is not a direct product. Then $E$ cannot contain a hyperelliptic involution.
\end{cor*}

In Subsection~\ref{subsec:embcyc}, we provide a bound on the order of elements of an $E < \map(S_g)$ such that $E \cong A_n  \rtimes \Z_m$. 

\begin{cor*}
Let $E < \map(S_g)$ such that $E \cong A_n  \rtimes \Z_m$, where $n \geq 7$. If $F \in E$, then: 
\begin{enumerate}[(a)]
\item $|F| \leq g$, and consequently
\item $F$ cannot be irreducible.
\end{enumerate}
\end{cor*}

In Subsection~\ref{subsec:boundcyc}, we derive a bound on the order of a liftable periodic mapping class under certain conditions. 
\begin{prop*}
Let $E < \map(S_g)$ such that $E \cong A_n \rtimes \Z_m$, where $n \geq 5$ and $n \neq 6$.  If $|E| > 5g-5$,  then $m \leq 26$.
\end{prop*}
In the next subsection, we digress a little from our central theme to answer a related problem pertaining to the existence of certain subgroups of $\map(S_g)$ that are isomorphic to $A_n \rtimes \Z$. In particular, we show that:
\begin{prop*}
For $g \geq 2$ and $n \geq 5$, consider a subgroup $H <  \map(S_g)$ such that $H \cong A_n \rtimes \Z$, where the $\Z$-component is generated by a multitwist $\prod_{j=1}^\ell T_{c_j}^{p_j}$, where $p_j \neq 0$ for all $j$ and $c_{j_1} \neq c_{j_2}$ when $j_1\neq j_2$  Then  $\ell \geq n$. In particular, the $\Z$-component of $H$ cannot be generated by a power of a Dehn twist.
\end{prop*}
\noindent Finally, in Section~\ref{sec:table} we provide a complete classification of the weak conjugacy classes of subgroups of $\map(S_g)$ isomorphic to $A_n \rtimes \Z_m$, for $g  = 10,11$. 
 
\section{Preliminaries} 
\label{sec:prelims}
In this section, we introduce some basic concepts and results that are pivotal to the theory we develop.
\subsection{Weakly conjugate actions on surfaces}
A finite group $H$ \textit{acts on} $S_g$ if there exists a monomorphism $\epsilon: H \hookrightarrow \homeo(S_g)$. We identify $H$ with its image $\epsilon(H)$ if no confusion caused. It is known that $H$ can always be realized as isometries for some hyperbolic metric on $S_g$ and the quotient $\Orb_H:=S_g/H$ can be identified with $\mathbb{H}/\Gamma$ for some cocompact Fuchsian group $\Gamma$. A cocompact Fuchsian group $\Gamma$ admits a presentation given by:
\[\Gamma=\biggl\langle~ \alpha_1,\beta_1,..., \alpha_{g_{0}},\beta_{g_{0}},\xi_{1},...,\xi_{r}~\bigg|~\xi_{1}^{m_1},...,\xi_{r}^{m_r},\prod_{j = 1}^{r} \xi_{j}\prod_{i = 1}^{g_{0}} [\alpha_{i},\beta_{i}]~\biggr\rangle. \tag{1} \label{eqn:eqn1}
\]
The \textit{signature} of $\Gamma$ is defined to be the tuple, 
$\sig(\Gamma):=(g_0; m_1, \dots, m_r).$
The quotient space $\mathcal{O}_H = \mathbb{H}/\Gamma$ is called a \textit{hyperbolic orbifold of genus $g_0$} with \textit{orbifold signature} $\sig(\mathcal{O}_H) := \sig(\Gamma$). The orbifold fundamental group of $\Orb_H$, denoted by $\pi_1^{orb}(\Orb_H)$, is defined to be the Fuchsian group $\Gamma$. The following version of the Riemann's existence theorem~\cite[Corollary 5.9.5]{Jones}, due to Harvey~\cite{harvey}, characterizes finite group actions on surfaces. 
\begin{theorem}
\label{thm:riemann}
A finite group $H$ acts on $S_g$ with a quotient orbifold $\Orb_H$ with $\sig(\Orb_H) = (g_0; m_1,\dots, m_r)$ if and only if the following conditions hold.
\begin{enumerate}[(i)]

    \item There exists a co-compact Fuchsian group $\Gamma$ with $\sig(\Gamma) = \sig(\Orb_H)$.
    \item There exists a surjective homomorphism
$$\phi_H : \Gamma \to H,$$ which is order-preserving on torsion elements.
    \item The following Riemann-Hurwitz equation holds: 
    $$\frac{2-2g}{|H|}= 2-2g_0-\sum_{j=1}^{r} (1-\frac{1}{m_j}). $$
\end{enumerate}
\end{theorem}
\noindent The epimorphism $\phi_H$, in Theorem \ref{thm:riemann}, is called a \textit{surface kernel map} associated to $H$. Two $H$-actions $H_1,H_2 < \homeo(S_g)$ are said to be \textit{conjugate} if $H_1 \cong H_2 \cong H$ and $H_1$ is conjugate to $H_2$ in $\homeo(S_g)$. We will also require the following result from the theory of Riemann surfaces~\cite[Lemma 11.5]{breuer}. 

\begin{lem}
\label{lem:subaction}
Let $H<\homeo(S_g)$ be of finite order with $\sig(\mathcal{O}_H)=(g_0;m_1,\ldots,m_r)$, and let $\G\in H$ be of order $m$. Then for $u\in \Z_m^{\times}$, we have
$$|\mathbb{F}_{\G}(u,m)|=|C_H(\G)|\cdot\sum\limits_{\substack{1\leq j\leq r, ~m\mid m_j \\ \G \sim_H \phi_H(\xi_j)^{m_ju/m}
}}\frac{1}{m_j},  $$
where $\mathbb{F}_{\G}(u,m)$ denote the set of fixed points of $\G$ with induced rotation angle $2\pi u/m$ and $C_H(\G)$ denotes the centralizer of $\G$ in $H$.
\end{lem}

Let $\Gamma$ be a co-compact Fuchsian group having the fixed presentations as in Equation (\ref{eqn:eqn1}). Following the notation in Equation (\ref{eqn:eqn1}), we state a result of Zieschang \cite[Theorem 5.8.2]{zieschang}, interpreted for orientation-preserving maps (see also \cite{broughton_normalizer}).
\begin{theorem}
\label{thm:zieschang}
Let $\psi:\Gamma \to \Gamma$ be an automorphism. Then $\psi(\xi_j)$ is conjugate to $\xi_{\pi(j)}$ for some permutation $\pi$ on $r$ letters.

\end{theorem}
\noindent We will use this result in the proofs of our main theorems.

Due to Theorem \ref{thm:riemann}, one can encode a cyclic action, combinatorially, as a tuple called a \textit{cyclic data set}. 
\begin{defn}\label{defn:data_set}
A \textit{data set of degree $n$} is a tuple
$$
D = (n,g_0; (c_1,n_1),\ldots, (c_r,n_r)),
$$
where $n\geq 2$ and $g_0 \geq 0$ are integers, and each $c_j \in \Z_{n_j}^\times$ such that:
\begin{enumerate}[(i)]
\item each $n_j\mid n$,
\item $\lcm(n_1,\ldots ,\widehat{n_j}, \ldots,n_{r}) = N$, for $1 \leq j \leq r$, where $N = n$ if $g_0 = 0$,  and
\item $\displaystyle \sum_{j=1}^{r} \frac{n}{n_j}c_j \equiv 0\pmod{n}$.
\end{enumerate}
The number $g$ determined by the Riemann-Hurwitz equation
\begin{equation}\label{eqn:riemann_hurwitz}
\frac{2-2g}{n} = 2-2g_0 + \sum_{j=1}^{r} \left(\frac{1}{n_j} - 1 \right) 
\end{equation}
is called the \textit{genus} of the data set, denoted by $g(D)$.
\end{defn}

\noindent To simplify notation, if a pair $(c_j,m_j)$ in the data set $D$ occurs with multiplicity $k_j$, then we denote it by $(c_j,m_j)^{[k_j]}$. By a result of Nielsen~\cite{nielsen_periodic}, this data characterizes cyclic actions up to conjugacy. 
\begin{theorem}
\label{thm:cyc_conj_class}
The cyclic data sets of genus $g$ and degree $n$ correspond to conjugacy classes
of periodic maps of order $n$ on $S_g$.
\end{theorem}
\noindent By the Nielsen realization theorem \cite{nielsen,kerckhoff}, cyclic data sets of genus $g$ and degree $n$ are also in one-to-one correspondence with the conjugacy classes of order $n$ periodic mapping classes in $\map(S_g)$. Therefore, we represent the conjugacy class of a periodic $F\in\map(S_g)$ by $D_{F}$. 
Although, conjugacy is a natural equivalence on the set of finite group actions on surfaces, classification of the actions up to conjugacy is absent in the literature for the family of non-abelian groups due to its known complexity. However, we have defined a weaker notion of conjugacy in \cite{dey} which we record here as well.

\begin{defn}\label{defn:wk_conjugacy}
Consider $H_1,H_2 < \homeo(S_g)$ such that $H_1 \cong H_2 \cong A_n \rtimes \Z_m$. Then we say that the actions of $H_1$ and $H_2$ on $S_g$ are \textit{weakly conjugate} if there exists an isomorphism $\psi: \pi_1^{orb}(\Orb_{H_1}) \to\pi_1^{orb}(\Orb_{H_2})$ and an isomorphism $\chi : H_1 \to H_2$ such that $(\chi\circ\phi_{H_1})(g)$ is conjugate to $(\phi_{H_2}\circ\psi)(g)$ in $H_2$, whenever $g\in\pi_1^{orb}(\Orb_{H_1})$ is of finite order.
\end{defn}
\noindent In other words, the action of $H_1$ and $H_2$ on $S_g$ are weakly conjugate if the following diagram commutes up to conjugacy at the level of torsion elements.
\begin{center}\label{fig:wkconjugacy}
\begin{tikzcd}[sep=12mm,
arrow style=tikz,
arrows=semithick,
diagrams={>={Straight Barb}}
                ]
\pi_1^{orb}(\Orb_{H_1}) \dar["\phi_{H_1}" '] \rar["\psi",""name=U]
        & \pi_1^{orb}(\Orb_{H_2}) \dar["\phi_{H_2}"]      \\
H_1\rar["\chi"  ,""name=D] & H_2
                     \ar[to path={(U) node[pos=.5,C] (D)}]{}
    \end{tikzcd}
\captionof{figure}{The weak conjugacy of the $H_i$-actions on $S_g$.}
\end{center}

\subsection{Liftability of periodic mapping classes under alternating covers} Consider an $H$-action on $S_g$, where $H \cong A_n$, and the induced branched cover $p: S_g \to \Orb_H (:= S_g/H)$. Suppose $E_{n,m}$ is an extension of $\Z_m$ by $A_n$, i.e., $A_n \lhd E_{n,m}$ and $E_{n,m}/A_n \cong \Z_m$. Then every $E_{n,m}$-action on $S_g$ is realized by lifting an order $m$ cyclic subgroup of $\map(\Orb_H)$ under $p$. A natural question in this context is when does a periodic mapping class lift under the branched cover $p$. In this paper, we give an answer to this question, up to conjugacy, through classifying weak conjugacy classes of $E_{n,m}$-action on $S_g$. 

Note that the mapping classes in $\map(\Orb_H)$ are represented by homeomorphisms that preserve the set of cone points in $\Orb_H$ and their orders. The Birman-Hilden theorem~\cite{birman,margalit-winarski} asserts that the sequence 
\begin{equation}
\label{eqn:eqn4}
1 \to H \to \mathrm{SMod}(S_g) \to \mathrm{LMod}(\Orb_H) \to 1
\end{equation}
is exact. For a periodic mapping class $\overline{G} \in \mathrm{LMod}(\Orb_H)$, the above short exact sequence restricts abstractly to $$1 \to A_n \to E_{n,m} \to \Z_m \to 1.$$

Our goal is to characterize weak-conjugacy classes of $E_{n,m}$-actions on $S_g$. Before going into the main section, we want to emphasize some properties of $E_{n,m}$ groups  which will be crucial in our context, but whose proofs are not readily available in literature.
\subsection{Properties of $E_{n,m}$ groups}
Since $Z(A_n)=\{id\}$, where $id$ is the identity permutation, it is known (see \cite[Chapter 11]{robinson}) that when $n\geq 5$ and $n \neq 6$, the isomorphism classes of the extensions of $\Z_m$ by $A_n$ are in bijective correspondence with the number of homomorphisms (also known as coupling) $\Phi: \Z_m \to \out(A_n)$. Hence one can conclude the following.
\begin{prop}\label{prop:extension}For $n\geq5$, $n\neq6$ and $m\geq2$, the short exact sequence: $1\to A_n \to E_{n,m} \to \Z_m \to 1$ always splits. Furthermore:
\begin{enumerate}[(i)]
\item  If $m$ is odd, $E_{n,m}$ is isomorphic to $A_n \times \Z_m$.
\item  If $m$ is even, $E_{n,m}$ is isomorphic to either $A_n \times \Z_m$ or $A_n \rtimes_{\phi} \Z_m$,
where $\varphi:\Z_m \to \aut(A_n)$ is a homomorphism determined by $\bar{1} \mapsto \chi_{(1~2)}$ with $\chi_{(1~2)}:=\sigma \mapsto (1~2) \sigma (1~2)$.
\end{enumerate}
\end{prop}
\noindent We add a flag variable $i\in \{0,1\}$, and assume that $\E$ is the direct product $A_n \times \Z_m$ if $i=0$, the semi-direct product $A_n \rtimes_{\varphi} \Z_m$ if $i=1$. To compute the automorphism group of $\E$, we use the result \cite[Theorem 1]{curran}, which gives the subgroup of the full automorphism group of a semi-direct product that fixes the normal component of the semi-direct product. It is easy to check that $A_n$ is the commutator subgroup of $\E$ thus, $A_n$ is characteristic. Therefore, in our context,  the subgroup mentioned in \cite{curran} coincides with the full automorphism group and we have the following result.

\begin{prop}\label{aut_semidirect}
For $n\geq5$, $n\neq6$ and $m\geq2$, we have the following isomorphisms of groups.
\begin{enumerate}[(i)]
\item Let $\beta:\Z_m\to A_n$ be defined by $$\beta(\bar{x})=\begin{cases}
    \beta(\bar{0})=id, & \text{if } x \equiv y \pmod{m}, \text{ where } 1\leq y \leq m-1 \text{ is even, and } \\
    \tau_\alpha (1\,2) \tau_\alpha^{-1} (1\, 2), & \text{ otherwise, where } \alpha \text{ is the conjugation map by } \tau_\alpha\in\Sigma_n.
    \end{cases}$$ Then
$\aut(A_n \rtimes_{\varphi} \Z_m)\cong
\Biggl\{
\begin{pmatrix}
\alpha & \beta\\
0 & \delta
\end{pmatrix}
:
\begin{matrix*}[l]
\alpha \in \mathrm{Aut}(A_n), \delta \in \mathrm{Aut}(\Z_m)
\end{matrix*}
\Biggr\}.$
\item $\mathrm{Aut}(A_n \times \Z_m) \cong
\Biggl\{
\begin{pmatrix}
\alpha & 0 \\
0 & \delta
\end{pmatrix}
:
\begin{matrix}
\alpha \in \mathrm{Aut}(A_n),\\ \delta \in \mathrm{Aut}(\Z_m)
\end{matrix}
\Biggr\}\cong \mathrm{Aut}(A_n) \times \mathrm{Aut}(\Z_m) \cong \Sigma_n\times\Z_m^{\times}.$
\end{enumerate}
\end{prop}
\begin{proof}
For elements $x,y$ in a group $G$, let $x^y$ denote the element $yxy^{-1}$. By \cite[Theorem 1]{curran}, $\beta$ satisfies the following conditions:
\begin{enumerate}[(a)]
\item $\beta(\bar{x}+\bar{y})=\beta(\bar{x})\beta(\bar{y})^{\delta(\bar{x})}$, for all $\bar{x},\bar{y} \in \Z_m$, and 
\item $\alpha(\sigma^{\bar{x}})=\alpha(\sigma)^{\beta(\bar{x})\delta(\bar{x})},$ for all $\sigma\in A_n$ and $\bar{x} \in \Z_m.$
\end{enumerate}
Taking $\bar{x}=\bar{y}=\bar{0}$, condition (a) implies $\beta(\bar{0})=id$. Since $\mathrm{Aut}(A_n) \cong \Sigma_n$ and $\mathrm{Aut}(\Z_m) \cong \Z_m^{\times}$, any $\alpha\in \mathrm{Aut}(A_n)$ can be identified with the conjugation map by some $\tau_\alpha \in \Sigma_n$ and any $\delta \in \mathrm{Aut}(\Z_m)$ with the multiplication by some $\ell$ such that $(\ell,m)=1$. Replacing $\bar{x}$ by $\bar{1}$ in condition (b), we get 
$$\begin{array}{lrcl}
       &\alpha(\sigma^{\bar{1}})
       &=&\alpha(\sigma)^{\beta(\bar{1})\delta(\bar{1})}, ~\forall \sigma \in A_n,\\
     \iff & \alpha(\phi(\bar{1}) \sigma) = \alpha((1~2)\sigma(1\,2))  &=& \beta(\bar{1}) (1\,2) \cdot\alpha(\sigma)\cdot (1~2) \beta(\bar{1})^{-1}, \,\forall \sigma \in A_n,\\
      \iff & \tau_\alpha(1\,2)\sigma(1\, 2)\tau_\alpha^{-1}  &=& \beta(\bar{1}) (1~2) \cdot\tau_\alpha\sigma\tau_\alpha^{-1}\cdot (1\,2) \beta(\bar{1})^{-1}, \,\forall \sigma \in A_n, \text{ and}\\
      \iff & \sigma  &=& (1\,2)\tau_\alpha^{-1}\beta(\bar{1}) (1\,2) \tau_\alpha \cdot \sigma \cdot\tau_\alpha^{-1}(1\,2) \beta(\bar{1})^{-1}\tau_\alpha(1\,2), \, \forall \sigma \in A_n.
     \end{array}$$
Hence $(1\,2)\tau_\alpha^{-1}\beta(\bar{1}) (1\,2) \tau_\alpha \in Z(A_n)=\{id\}$ implies that  $(1\,2)\tau_\alpha^{-1}\beta(\bar{1}) (1\,2) \tau_\alpha=id$. Thus $\beta(\bar{1})=\tau_\alpha (1\, 2) \tau_\alpha^{-1} (1~2)$. By condition (a) we have

$$\beta(\bar{2})=\beta(\bar{1}+\bar{1})=\beta(\bar{1})\beta(\bar{1})^{\delta(\bar{1})}=\beta(\bar{1})(1\, 2)\beta(\bar{1})(1\,2)=id.$$
Similarly,
$$\beta(\bar{3})=\beta(\bar{1}+\bar{2})=\beta(\bar{1})\beta(\bar{2})^{\delta(\bar{1})}=\beta(\bar{1}).$$
Therefore, by applying condition (a) inductively, we have
$$\beta(\bar{x})=\begin{cases}
    \beta(\bar{0})=id, & \text{if }x \equiv y \pmod{m}, \text{ where } 1\leq y \leq m-1 \text{ is even, and } \\
    \tau_\alpha (1\,2) \tau_\alpha^{-1} (1\, 2), & \text{otherwise.}
    \end{cases}$$
which proves (i). 

For (ii), note that, conditions (a) and (b) imply that $\beta:\Z_m\to A_n$ is a homomorphism with $\beta(\Z_m)\subseteq Z(A_n)=\{id\}$. Therefore $\beta=0$ (trivial) and we have our desired result.

\end{proof}

\begin{rem}
It is worth mentioning that the groups $\aut(A_n \rtimes_{\varphi} \Z_m)$ and $\aut(A_n \times \Z_m)$ are isomorphic via the map:
$$ \begin{pmatrix}
\alpha & \beta\\
0 & \delta
\end{pmatrix} \mapsto
\begin{pmatrix}
\alpha & 0\\
0 & \delta
\end{pmatrix}.   $$
\end{rem}

\noindent Now we show that the groups $\E$ are two-generated. First we need the following lemma from \cite{miller}.
\begin{lem}\label{lem:generators} If $\ell_1,\ell_2$ represent a pair of numbers, each being greater than unity, such that neither exceeds $n$ but their sum exceeds $n$, then it is always possible to find two cycles of
orders $\ell_1$ and $\ell_2$ respectively such that they generate the alternating group of degree $n$ whenever both $\ell_1$ and $\ell_2$ are odd.
\end{lem}
\begin{prop}\label{prop:gen_direct}
The group $A_n\times\Z_m$ is two-generated.
\end{prop}
\begin{proof}
For $n$ odd, applying Lemma \ref{lem:generators}, there exist generators $\sigma_1,\sigma_2$ of $A_n$ of orders $n-2,n$, respectively. Then, we know that $A_n\times\Z_m$ is generated by $(\sigma_1,\bar{0}),(\sigma_2,\bar{0})$, and $(id,\bar{1})$. Now, consider the elements $(\sigma_1,\bar{1})$ and $(\sigma_2,\bar{1})$ in $A_n\times\Z_m$. Note that $(\sigma_1,\bar{1})^{n-2}=(id,\overline{n-2})$ and $(\sigma_2,\bar{1})^{n}=(id,\overline{n})$. Moreover, $\langle \overline{n-2}, \bar{n}\rangle =\Z_m$ since $(n-2,n)=1$. Hence there exist integers $k$ and $\ell$ such that $(id,\overline{n-2})^k \cdot (id,\overline{n})^{\ell}=(id,\bar{1})$. Furthermore, $(\sigma_1,\bar{1}) \cdot (id,-\bar{1})=(\sigma_1,\bar{0})$ and $(\sigma_2,\bar{1}) \cdot (id,-\bar{1})=(\sigma_2,\bar{0})$. Thus, $(\sigma_1,\bar{1})$ and $(\sigma_2,\bar{1})$ also generate $A_n\times\Z_m$.

For $n$ even, applying Lemma \ref{lem:generators}, there exist generators $\sigma_1,\sigma_2$ of $A_n$ of orders $n-3,n-1$, respectively. By similar arguments, as above, one can show that $(\sigma_1,\bar{1})$ and $(\sigma_2,\bar{1})$ also generate $A_n\times\Z_m$. Hence, $A_n\times\Z_m$ is two-generated.
\end{proof}

To show that the group $E_{n,m}^1$ is two-generated, we need the following elementary number-theoretic lemma, which is a consequence of \cite[Corollary 3]{BL} for $n \geq 21$ and for lower values of $n$, it can be deduced by a direct computation.
\begin{lem}\label{lem:primes}
For $n\geq 5$, there are at least two distinct primes $p_1$ and $p_2$ such that $n \leq p_1 < p_2 < 2n$.
\end{lem} 

\begin{prop}\label{prop:gen_semidirect}
$A_n\rtimes_{\varphi}\Z_{m}$ is two-generated.
\end{prop}
\begin{proof}
\textit{Case (i): $6\nmid m$.} Taking $\sigma=(3~4~5)$, we have that $\sigma':=[\sigma\cdot(1~2)]^{m}$ equals $(3~5~4)$ or $(3~4~5)$, depending upon whether $m=6k+2$ or $6k+4$, respectively. Taking $\tau=(1~3~5~2~4)^{-1}\cdot \tau_a \cdot (1~3~5~2~4)$ if $m=6k+2$ and $\tau=(1~3~5~2~4)\cdot \tau_a \cdot (1~3~5~2~4)^{-1}$ if $m=6k+4$ where 
$$\tau_a = 
 \begin{cases}
    (1~2~\dots~n), & \text{if } n \text{ is odd, and} \\
    (2~3~\dots~n), &  \text{if } n \text{ is even}
  \end{cases}$$
we have $A_n=\langle \sigma',\tau\rangle$. The group $\langle (\sigma, \bar{1}), (\tau, \bar{0}) \rangle \cong A_n\rtimes_{\varphi}\Z_{m}$. Hence, $A_n\rtimes_{\varphi}\Z_{m}$ is two-generated whenever $6\nmid m$.

\textit{Case (ii): $6\mid m$.} Take two distinct odd primes $p_1$ and $p_2$ such that $p_1,p_2 \leq n$ but $p_1+p_2>n$ (existence is guaranteed by Lemma \ref{lem:primes}). It is apparent that one of them must be $\leq n-2$, and we may assume that this prime is $p_1$. Suppose that $m=6k=p_1^{\ell_1}\cdot 2p_2^{\ell_2}\cdot q$ where $\ell_1,\ell_2\geq0$ such that $(p_1,q)=(p_2,q)=1$. Now consider a $p_1$-cycle $\sigma_1$, and a $p_2$-cycle $\sigma_2$ such that $A_n=\langle \sigma_1,\sigma_2\rangle$ as guaranteed by Lemma \ref{lem:generators}. Without loss of generality, as length$(\sigma_1)\leq n-2$, assume that $\sigma_1$ and $(1~2)$ are disjoint (if not, conjugate both the generators by a fixed suitable permutation to achieve the same).

It remains to be shown that the elements $(\sigma_1,\bar{p}_1^{\ell_1})$ and $(\sigma_2,\bar{2}\bar{p}_2^{\ell_2})$ generate $A_n\rtimes_{\varphi}\Z_{m}$. We note that $(\sigma_1,\bar{p}_1^{\ell_1})^{2p_2^{\ell_2}\cdot q}$ is $(\sigma_1^{2p_2^{\ell_2}\cdot q},\bar{0})$ and $(\sigma_2,\bar{2}\bar{p}_2^{\ell_2})^{p_1^{\ell_1}q}=(\sigma_2^{p_1^{\ell_1}q},\bar{0})$. But $\langle \sigma_1^{2p_2^{\ell_2}\cdot q} \rangle=\langle\sigma_1\rangle$ and $\langle \sigma_2^{p_1^{\ell_1}q} \rangle = \langle \sigma_2 \rangle$. Hence $(\sigma,\bar{0}) \in \Big\langle (\sigma_1,\bar{p}_1^{\ell_1}),(\sigma_2,\bar{2}\bar{p}_2^{\ell_2}) \Big\rangle$ for all $\sigma\in A_n$. Therefore $$(id,\bar{p}_1^{\ell_1}),(id,\bar{2}\bar{p}_2^{\ell_2})\in \Big\langle (\sigma_1,\bar{p}_1^{\ell_1}),(\sigma_2,\bar{2}\bar{p}_2^{\ell_2}) \Big\rangle.$$ Since $(p_1^{\ell_1},2p_2^{\ell_2})=1$, $(id,\bar{1})\in \Big\langle (\sigma_1,\bar{p}_1^{\ell_1}),(\sigma_2,\bar{2}\bar{p}_2^{\ell_2}) \Big\rangle$. Thus, we have $$\Big\langle (\sigma_1,\bar{p}_1^{\ell_1}),(\sigma_2,\bar{2}\bar{p}_2^{\ell_2}) \Big\rangle=A_n\rtimes_{\varphi}\Z_{m}.$$
Hence, $A_n\rtimes_{\varphi}\Z_{m}$ is two-generated for $6\mid m$.
\end{proof}
We end this section with the following final remark.
\begin{rem}
	It is evident that $(\sigma_1,\bar{x}_1)$ is conjugate to $(\sigma_2,\bar{x}_2)$ in $A_n\times\Z_m$ if and only if $\sigma_1$ is conjugate to $\sigma_2$ in $A_n$ and $\bar{x}_1=\bar{x}_2$. For the group $A_n\rtimes_{\varphi}\Z_m$, an element $(\sigma_1,\bar{x}_1)$ is conjugate to $(\sigma_2,\bar{x}_2)$ in $A_n\rtimes_{\varphi}\Z_m$ if and only if
	$\bar{x}_1=\bar{x}_2$ and there exists $(\tau,\bar{y}) \in A_n\rtimes_{\varphi}\Z_m$ such that:
	\begin{center}
		$\sigma_2=\begin{cases}
		\tau\cdot\sigma_1\cdot\tau^{-1}, & \text{if both}~x_1~\text{and}~y~\text{are even,} \\
		\tau\cdot (1~2)\sigma_1(1~2)\cdot\tau^{-1}, & \text{if}~x_1~\text{is even but}~y~\text{is odd,}\\
		\tau\cdot \sigma_1\cdot(1~2)\tau^{-1}(1~2), & \text{if}~x_1~\text{is odd but}~y~\text{is even, and} \\
		\tau\cdot(1~2)\sigma_1\cdot\tau^{-1}(1~2), & \text{if both}~x_1~\text{and}~y~\text{are odd.}
		\end{cases}$
	\end{center}
	
\end{rem} 

\section{Main Theorems} 
\begin{defn}
\label{defn:H_data_set}
An \textit{$\E$-data set $\D_e$ of degree $(n,m)$ and genus $g\geq 2$} is an ordered tuple 
$$\D_e=\bigl((n,m,i),g_0; [(\sigma_1,\bar{x}_1);m_1,t_1], \dots, [(\sigma_r,\bar{x}_r);m_r,t_r]\bigr)$$
where $n,m,g_0$ are integers with $n\geq 5$ but $n\neq6$, $m \geq 1$, $g_0 \geq 0$, and $(\sigma_i,\bar{x}_i) \in \E$, non-trivial, satisfying the following conditions:
   \begin{enumerate}[(i)]
   \item  $\displaystyle  \frac{2-2g}{m \cdot n\, !/2}=  2-2g_{0}-\sum_{j=1}^{r} \left(1-\frac{1}{m_j}\right).$
     \item For $1 \leq j \leq r,$ $m_j$ is the order of $(\sigma_j,\bar{x}_j)$ in $\E$ and $t_j$ is the order of $\bar{x}_j$ in $\Z_m$.
\item For $g_0=0$, we have:
\begin{enumerate}[(a)]
     \item $\prod_{j = 1}^{r} (\sigma_j,\bar{x}_j)= (id,\bar{0})$, and
    \item $\bigl\langle (\sigma_1,\bar{x}_1),\ldots,(\sigma_r,\bar{x}_r)\bigr\rangle = \E$.
\end{enumerate}
\item For $g_0=1$, there exist $(\sigma_{r+1},\bar{x}_{r+1}), (\sigma_{r+2},\bar{x}_{r+2}) \in \E$ such that:
\begin{enumerate}[(a)]
    \item $\prod_{j = 1}^{r} (\sigma_j,\bar{x}_j)= [(\sigma_{r+2},\bar{x}_{r+2}), (\sigma_{r+1},\bar{x}_{r+1})]$, and
    \item $\bigl\langle (\sigma_1,\bar{x}_1),\ldots,(\sigma_r,\bar{x}_r), (\sigma_{r+1},\bar{x}_{r+1}), (\sigma_{r+2},\bar{x}_{r+2}) \bigr\rangle = \E$.
\end{enumerate}
\item For $g_0\geq 2$, the element $\sum_{j=1}^{r} \bar{x}_j=\bar{0}$.
        \end{enumerate}

\end{defn}
\noindent If a tuple $[(\sigma_j,\bar{x}_j);m_j,t_j]$ in $\D_e$ occurs more than once, then we will use the symbol $[(\sigma_j,\bar{x}_j);m_j,t_j]^{[\ell_j]}$ to denote that it occurs with multiplicity $\ell_j$ in $\D_e$. 

\begin{defn}
\label{defn:alt_sym_data_set}
Let $\D_e$ be an $E_{n,m}^i$-data set as in Definition~\ref{defn:H_data_set}. Then: 
\begin{enumerate}[(a)]
\item $\D_e$ is called an \textit{alternating data set} if, in which case $\D_e$ will take the simpler form: 
$$\D_e=\bigl(n,g_0; [\sigma_1;m_1], \dots, [\sigma_r;m_r]\bigr).$$
\item $\D_e$ is called a \textit{symmetric data set} if $m=2$ and $i=1$, in which case $\D_e$ will take the form:
$$\D_e=\bigl((n,2,1),g_0; [(\sigma_1,\bar{x}_1);m_1,t_1], \dots, [(\sigma_r,\bar{x}_r);m_r,t_r]\bigr).$$
\end{enumerate}  
\end{defn}
\noindent From here on, we will fix the notation $\D_a$ for an alternating data set and $\D_s$ for a symmetric data set. We will now define an equivalence on $\E$-data sets.
\begin{defn}
\label{defn:eq_data_sets}
Two $\E$-data sets 
\begin{gather*}
\D_e=\bigl((n,m,i),g_0; [(\sigma_1,\bar{x}_1);m_1,t_1], \dots, [(\sigma_r,\bar{x}_r);m_r,t_r]\bigr)\\
\text{and}\\
\D_e'=\bigl((n',m',i'),g_0'; [(\sigma_1',\bar{x}_1');m_1',t_1'], \dots, [(\sigma_{r'}',\bar{x}_{r'}');m_{r'}',t_{r'}']\bigr)
\end{gather*}  are said to be \textit{equivalent} if there exists a permutation $\pi$ on $r$ letters and an integer $\ell$ with $(\ell,m)=1$, satisfying the following conditions.
\begin{enumerate}[(i)]
    \item $n=n', m=m', i=i', g_0=g_0',$ and $r=r'$.
    \item For $1 \leq j \leq r$, $\bar{x}_{\pi(j)}'=\ell\bar{x}_j$ in $\Z_m$.
    \item For $i=0$, there exists $\tau_{\alpha} \in \Sigma_n$ such that $\sigma_{\pi(j)}'$ is conjugate to $\tau_\alpha \sigma_j \tau_{\alpha}^{-1}$ in $A_n$.
    \item For $i=1$, there exists $\tau_{\alpha} \in \Sigma_n$ and $\beta:\Z_m\to A_n$ as defined in Proposition \ref{aut_semidirect}, such that $\sigma_{\pi(j)}'$ is conjugate to $\tau_\alpha \sigma_j \tau_{\alpha}^{-1}\beta(\bar{x}_j)$ in $A_n$.

\end{enumerate}
\end{defn}

\noindent We represent the equivalence class of an $\E$-data set $\D_e$ by $[\D_e]$.
\begin{prop}
\label{prop:wlp}
 Equivalence classes of $\E$-data sets of genus $g$ correspond to weak conjugacy classes of $\E$-actions on $S_g$.
\end{prop}
\begin{proof}

We show this correspondence by establishing a natural bijective map:
\begin{center}
$\left\{\parbox{38mm}{\raggedright Equivalence classes of $\E$-data sets}\right\}\xrightarrow{\makebox[1.5cm]{$\Theta_e$}}
\left\{\parbox{40mm}{\raggedright Weak conjugacy classes of $\E$-actions on $S_g$} \right\}.$
\end{center}
Consider an $\E$-data set $\D_e=\bigl((n,m,i),g_0; [(\sigma_1,\bar{x}_1);m_1,t_1], \dots, [(\sigma_r,\bar{x}_r);m_r,t_r]\bigr)$.  Fix a co-compact Fuchsian group $\Gamma$ with $\sig(\Gamma)=(g_0;m_1,\dots,m_r)$, having the presentation as in (\ref{eqn:eqn1}). Define a map $\phi: \Gamma \to \E$ as follows:
$$\begin{array}{rcll}
\phi(\xi_j) & := & (\sigma_j,\bar{x}_j), & \text{for } g_0 \geq 0 \text{ and } 1\leq j \leq r,\\
\phi(\alpha_1) ,\phi(\beta_1) & := & (\sigma_{r+1},\bar{x}_{r+1}), (\sigma_{r+2},\bar{x}_{r+2}) &  \text{for }  g_0 = 1,\\
\phi(\alpha_1),\phi(\beta_1), \phi(\alpha_2),\phi(\beta_2) & := & (\sigma, \bar{x}),(\tau, \bar{y}),(\rho_1,\bar{0}),(\rho_2,\bar{0}) & \text{for }  g_0 \geq 2, \text{ and}\\
\phi(\alpha_i)=\phi(\beta_i) & := & (id,\bar{0}) & \text{for }  g_0 \geq 2 \text{ and } 3 \leq i \leq g_0,
\end{array}$$
such that $\langle (\sigma, \bar{x}),(\tau, \bar{y}) \rangle =E_{n,m}^{i}$ and $[(\rho_1,\bar{0}),(\rho_2,\bar{0})]\cdot[(\sigma, \bar{x}),(\tau, \bar{y})]\cdot\Pi_{i=1}^r (\sigma_i,\bar{x}_i) =(id,\bar{0})$. The existence of $(\sigma, \bar{x}),(\tau, \bar{y})$ is guaranteed by Propositions \ref{prop:gen_direct} and \ref{prop:gen_semidirect}. The existence of $(\rho_1,\bar{0}),(\rho_2,\bar{0})$ follows from the fact that every element in $A_n$ is a commutator \cite{Ore}. The conditions (i)-(v) on the data set $\D_e$ ensure that $\phi$ satisfies the hypothesis of Theorem \ref{thm:riemann}. Let $H<\homeo(S_g)$ be an  $E_{n,m}^{i}$-action thus obtained from the surface kernel map $\phi$. Define $\Theta_e([\D_e]):=[H]$.

Let $\D_e'=\bigl((n',m',i'),g_0'; [(\sigma_1',\bar{x}_1');m_1',t_1'], \dots, [(\sigma_{r'}',\bar{x}_{r'}');m_{r'}',t_{r'}']\bigr)$ be another $\E$-data set. Define a map, as described above, $\phi': \Gamma' \to \E$ where $\Gamma'$ is a co-compact Fuchsian group with $\sig(\Gamma')=(g_0';m_1',\dots,m_{r'}')$ and having the presentation as below:
\[\Gamma'=\biggl\langle~ \alpha_1',\beta_1',..., \alpha_{g_{0}}',\beta_{g_{0}}',\xi_{1}',...,\xi_{r}'~\bigg|~\xi_{1}'^{m_1'},...,\xi_{r}'^{m_r'},\prod_{j = 1}^{r} \xi_{j}'\prod_{i = 1}^{g_{0}} [\alpha_{i}',\beta_{i}']~\biggr\rangle.\tag{$2$} \label{eqn:eqn2}
\]
Let $\Theta_e([\D_e'])=[H']$. Suppose $[\D_e]=[\D_e']$, i.e., there exist permutations $\pi,\tau_\alpha \in \Sigma_r$, an integer $\ell$ with $(\ell,m)=1$, and $\beta:\Z_m\to A_n$ satisfying the conditions (i)-(iv) as in Definition \ref{defn:eq_data_sets}. It follows from basic group theoretic fact that, for the permutation $\pi$, there exists a $\psi'\in \aut(\Gamma')$ such that $\psi'(\xi_i')$ is conjugate to $\xi_{\pi(i)}'$ in $\Gamma'$. Define an isomorphism $\psi:\Gamma \to \Gamma'$ given by $\psi(\xi_i):=\psi'(\xi_i')$ and an isomorphism $\chi \in \aut(E_{n,m}^{i})$ given by $\begin{pmatrix}
\alpha & \beta\\
0 & \delta
\end{pmatrix}$ where $\alpha$ is the conjugation map by $\tau_\alpha$ and $\delta$ is the multiplication map by $\ell$. Now one can easily verify that the diagram in Figure \ref{fig:wkconjugacy2} commutes up to conjugacy at the level of torsion elements. Hence $H$ and $H'$ are weakly conjugate, i.e., $[H]=[H']$. Thus $\Theta_e$ is well-defined.
\begin{center}
\begin{tikzcd}[sep=12mm,
arrow style=tikz,
arrows=semithick,
diagrams={>={Straight Barb}}
                ]
\Gamma \dar["\phi" '] \rar["\psi",""name=U]
        & \Gamma' \dar["\phi'"]      \\
\E\rar["\chi"  ,""name=D] & \E
                     \ar[to path={(U) node[pos=.5,C] (D)}]{}
    \end{tikzcd}
\captionof{figure}{A commutative diagram.}
\label{fig:wkconjugacy2}
\end{center}

Conversely, if $H$ and $H'$ are weakly conjugate, then by Definition \ref{defn:wk_conjugacy}, there exist $\psi$ and $\chi$ such that the diagram in Figure \ref{fig:wkconjugacy2} commutes up to conjugacy at the level of torsion elements, i.e., $\phi' \circ \psi(\xi_j)$ is conjugate to $\chi \circ \phi(\xi_j)$ in $\E$ for all $1\leq j \leq r$. Thus, by \cite[Theorem 5.8.2]{zieschang}, it follows that $\psi(\xi_j)$ is conjugate to $\xi_{\pi(j)}'$ for some $\pi \in \Sigma_r$. Hence $\phi' \circ \psi(\xi_j)$ is conjugate to $\phi'(\xi_{\pi(j)}')=\Bigl(\sigma_{\pi(j)}',\bar{x}_{\pi(j)}'\Bigr)$. Let $\chi =\begin{pmatrix}
\alpha & \beta\\
0 & \delta
\end{pmatrix} \in \aut(E_{n,m}^{i})$, then $$\chi \circ \phi(\xi_j)=\begin{pmatrix}
\alpha & \beta\\
0 & \delta
\end{pmatrix}(\sigma_j,\bar{x}_j)=\Bigl(\alpha(\sigma_j)\beta(\bar{x}_j),\delta(\bar{x}_j)\Bigr)=\Bigl(\tau_{\alpha}\sigma_j\tau_{\alpha^{-1}}\beta(\bar{x}_j),\ell\bar{x}_j\Bigr)$$ is conjugate to $\Bigl(\sigma_{\pi(j)}',\bar{x}_{\pi(j)}'\Bigr)$ in $\E$ for all $1\leq j \leq r$. Thus the conditions (i)-(iv) of Definition \ref{defn:eq_data_sets} are satisfied. Therefore, $\Theta_e$ is injective.

Finally, it remains to be shown that $\Theta_e$ is surjective. Let $H<\homeo(S_g)$ be an  $E_{n,m}^{i}$-action. Identifying $H$ with $\E$ and $\pi_1^{orb}(\Orb_H)$, where $\Orb_H=S_g/H$, with $\Gamma$, let $\phi_H:\Gamma \to \E$ be the surface kernel epimorphism, as in Theorem \ref{thm:riemann}. Then, one can check that $\Theta_e([\D_e])=[H]$ with $\D_e=\bigl((n,m,i),g_0; [(\sigma_1,\bar{x}_1);m_1,t_1], \dots, [(\sigma_r,\bar{x}_r);m_r,t_r]\bigr)$ such that $(\sigma_j,\bar{x}_j):=\phi_H(\xi_j)$ for all $1\leq j \leq r$. This completes the proof.

\end{proof}

\noindent As an immediate consequence, we have the following result.
\begin{cor}
\label{cor:alt_corr}
For $g\geq2$ and $n\geq5$ but $n\neq6$, equivalence classes of alternating data sets of genus $g$ and degree $n$ correspond to the weak conjugacy classes of $A_n$-actions on $S_g$.
\end{cor}

A final ingredient in the proof of our main result is the following proposition, which follows from Lemma~\ref{lem:subaction}. 

\begin{prop}
\label{prop:cyclicgen}
Let $\D_e=\bigl((n,m,i),g_0; [(\sigma_1,\bar{x}_1);m_1,t_1], \dots, [(\sigma_r,\bar{x}_r);m_r,t_r]\bigr)$ represent the weak conjugacy class of an $\E$-action on $S_g$. For $(\sigma,\bar{x}) \in \E$ with $|(\sigma,\bar{x})|=d,$ we have: $$\D_e[(\sigma,\bar{x})]=\left(d,\tilde{g_0};(u_{ij}^{-1},d_i)^{\bigl[\frac{d_i}{d} \cdot |\mathbb{f}_{(\sigma,\bar{x})^{d/d_i}}(u_{ij},d_i)|\bigr]}: u_{ij}\in\mathbb{Z}_{d_i}^{\times}, ~d_i~|~d\right),$$ 
where $\tilde{g_0}$ is determined by the  Riemann-Hurwitz equation, and $\mathbb{f}_{(\sigma,\bar{x})^{d/d_i}}(u_{ij},d_i)$ denotes the set of fixed points of $(\sigma,\bar{x})^{d/d_i}$ with induced rotation angle $2\pi u_{ij}^{-1}/d_i$ that do not come from fixed points of any ${(\sigma,\bar{x})}^{d/d_{i'}}$ with induced rotation angle $2\pi u_{i'j'}^{-1}/d_{i'}$ such that $d_i'\neq d_i,~d_i|d_i'|d$, $\gcd(u_{i'j'},d_{i'})=1$, and $u_{ij}\equiv u_{i'j'}(\mathrm{mod}~d_i)$. Further, the cardinality of the set $\mathbb{f}_{(\sigma,\bar{x})^{d/d_i}}(u_{ij},d_i)$ is given by the recursive formula:
$$|\mathbb{f}_{(\sigma,\bar{x})^{d/d_i}}(u_{ij},d_i)|= |\mathbb{F}_{(\sigma,\bar{x})^{d/d_i}}(u_{ij},d_i)|- \sum\limits_{\substack{d_i'\in \N \\d_i'\neq d_i\\d_i|d_i'|d}}\sum\limits_{\substack{ \gcd(u_{i'j'},d_{i'})=1\\u_{ij}\equiv u_{i'j'}(\mathrm{mod}~d_i)}}|\mathbb{f}_{(\sigma,\bar{x})^{d/d_{i'}}}(u_{i'j'},d_{i'})|~.$$
\end{prop}

\noindent Thus, we obtain the main result of this paper, which follows from Propositions~\ref{prop:wlp} - \ref{prop:cyclicgen}.
\begin{theorem}[Main Theorem 1]\label{thm:main1}
Let $F, G \in \map(S_g)$ be two periodic mapping classes. Then there exist conjugates $F', G'$ of $F , G$ respectively such that $\langle F', G'\rangle \cong \E$ if and only if there exists an $\E$-data set $\D_e$ of genus $g$ such that the cyclic data sets $D_{F}=\D_e[(\sigma,\bar{x})]$ and $D_{G}=\D_e[(\tau,\bar{y})]$ for some generating pair $(\sigma,\bar{x}),(\tau,\bar{y}) \in \E$. Moreover, if $\lambda_{n}$ denotes the number of non-equivalent $\E$-data sets of genus $g$ with this property, then $F$ and $G$ weakly generate $\lambda_n$ many $\E$-subgroups of $\map(S_g)$ up to weak conjugacy.
\end{theorem}

\begin{exmp}
Consider the weak conjugacy class of action represented by an $\E$-data set
$$\D_e=\bigl((7,10,0),1; [(\sigma_1,\bar{1});10,10],[(\sigma_2,\bar{1});70,10], [(id,-\bar{2});5,5]\bigr),$$
where $\sigma_1=(1~2~3~4~5)$ and $\sigma_2=(1~2~3~4~5~6~7)$. Note that $\langle (\sigma_1,\bar{1}),(\sigma_2,\bar{1})\rangle=A_7 \times \Z_{10}$. We now compute the cyclic factors $\D_e[(\sigma_1,\bar{1})]$ and $\D_e[(\sigma_2,\bar{1})]$ associated with the $\E$-data set $\D_e$ by applying Proposition~\ref{prop:cyclicgen}. For $(\sigma_1,\bar{1})$, first note that $d= |(\sigma_1,\bar{1})|=10$. To compute the multiplicity of $(u_{1j}^{-1},d_1)=(1,10)$ in $\D_e[(\sigma_1,\bar{1})]$, we use the recursive formula in Proposition~\ref{prop:cyclicgen} and obtain that
$|\mathbb{f}_{(\sigma_1,\bar{1})}(1,10)|= |\mathbb{F}_{(\sigma_1,\bar{1})}(1,10)|$. By Lemma~\ref{lem:subaction}, we have:
\[|\mathbb{F}_{(\sigma_1,\bar{1})}(1,10)|= |C_{A_7 \times \Z_{10}}((\sigma_1,\bar{1}))|\cdot \frac{1}{10}=50\cdot\frac{1}{10}=5.\]
Thus, the multiplicity of the pair $(1,10)$ in $\D_e[(\sigma_1,\bar{1})]= \frac{10}{10} \cdot 5=5$. Similarly, the multiplicities of the pairs $(u_{12}^{-1},d_1)=(3,10),(u_{13}^{-1},d_1)=(7,10),(u_{14}^{-1},d_1)=(9,10)$ can be computed and all of them are turned out to be zero. The next divisor of $d(=10)$ is $d_2=5$. To compute the multiplicity of $(u_{21}^{-1},d_2)=(1,5)$ in $\D_e[(\sigma_1,\bar{1})]$, we appeal to the same recursive formula, and obtain that
\begin{eqnarray*}
|\mathbb{f}_{(\sigma_1,\bar{1})}(1,5)| &=& |\mathbb{F}_{(\sigma_1,\bar{1})^{2}}(1,5)|-|\mathbb{f}_{(\sigma_1,\bar{1})}(1,10)|\\
&=& |C_{A_7 \times \Z_{10}}((\sigma_1^2,\bar{2}))|\cdot \frac{1}{10}-5\\
&=& 50\cdot\frac{1}{10}-5=0.
\end{eqnarray*}
Similar computations show that the multiplicities of the pairs $(u_{22}^{-1},d_2)=(2,5),(u_{23}^{-1},d_2)=(3,5),(u_{24}^{-1},d_2)=(4,5)$ are also zero. The next and the smallest (non-trivial) divisor of $d(=10)$ is $d_3=2$. We compute the multiplicity of $(u_{31}^{-1},d_3)=(1,2)$ in $\D_e[(\sigma_1,\bar{1})]$ in a similar manner, and first derive that:
\begin{eqnarray*}
|\mathbb{f}_{(\sigma_1,\bar{1})}(1,2)| &=& |\mathbb{F}_{(\sigma_1,\bar{1})^{5}}(1,2)|-\sum\limits_{\substack{u_{1j}=1,3,7,9}}|\mathbb{f}_{(\sigma_1,\bar{1})}(u_{1j},10)|-\sum\limits_{\substack{u_{2j}=1,3}}|\mathbb{f}_{(\sigma_1,\bar{1})}(u_{2j},5)|\\
&=& |C_{A_7 \times \Z_{10}}((id,\bar{5}))|\cdot \left(\frac{1}{10}+\frac{1}{70}\right)-5\\
&=& \frac{7!}{2}\cdot10\cdot\left(\frac{1}{10}+\frac{1}{70}\right)-5=2875.
\end{eqnarray*}
Thus, the multiplicity of the pair $(1,2)$ in $\D_e[(\sigma_1,\bar{1})]$ is $\frac{2}{10}\cdot 2875=575$. The genus of $\D_e$ is given by the equation
\[\frac{2-2g}{7!\cdot5}=2-2-\Bigl[\left(1-\frac{1}{10}\right)-\left(1-\frac{1}{70}\right)-\left(1-\frac{1}{5}\right)\Bigr],\]
which implies that $g=33841$. Finally, the orbifold genus of $\D_e[(\sigma_1,\bar{1})]$ is given by the equation
\[ \frac{2-2g}{10}=2-2g_0-\left(1-\frac{1}{10}\right)\cdot5-\left(1-\frac{1}{2}\right)\cdot575,\]
which implies that $g_0=3239$. Therefore $\D_e[(\sigma_1,\bar{1})]=(10,3239;(1,10)^{[5]},(1,2)^{[575]})$. For $(\sigma_2,\bar{1})$, following the same line of arguments as above, one can conclude that \[\D_e[(\sigma_2,\bar{1})]=(70,298;(11,70),(51,70),(1,10)^{[51]},(2,5)^{[360]},(1,2)^{[72]}).\]
Thus, any two periodic mapping classes $F$ and $G$ in $\map(S_g)$ weakly generate a subgroup isomorphic to $A_7 \times \Z_{10}$ if $D_F=\D_e[(\sigma_1,\bar{1})]$ and $D_G=\D_e[(\sigma_2,\bar{1})]$.
\end{exmp}

For an $H$-action on $S_g$, where $H \cong A_n$, any $\overline{G} \in \map(\mathcal{O}_H)$ induces a permutation, denoted by $\Pi_{\overline{G}}$, on the set of cone points of $\mathcal{O}_H$. Let $D_{\overline{G}}$ denote the  cyclic data set of $\overline{G}$, where $\overline{G}$ is viewed as an element of the mapping class group of the surface part of $\mathcal{O}_H$ (obtained by forgetting the cone points). 

\begin{defn}
\label{defn:wlp}
Let $\D_a$ be an alternating data set, $D$ be a cyclic data set, and $\Pi \in \Sigma_r$. Then the pair $(\D_a, (D,\Pi))$ is called a \textit{weak-liftable pair of degree $m$} if there exists an $H \in \D_a$ and a periodic mapping class $\overline{G} \in \map(\Orb_H)$ of order $m$ satisfying the following conditions.
\begin{enumerate}[(i)] 
\item $\overline{G}$ lifts under the branched cover $S_g \to S_g/H$.
\item The cyclic data set $D_{\overline{G}}=D$, and the permutation $\Pi_{\overline{G}}=\Pi$.  
\end{enumerate}
\end{defn}
\noindent Note that it is implicit from Definition~\ref{defn:wlp} that $|\Pi|$ divides $m$. Also, since $(D,\Pi)$ depend on $\overline{G}$, we will fix the notation $\wlp$ for a weak-liftable pair with the implicit assumption that $\overline{G}$ is indeed the representative that lifts under $S_g \to S_g/H$. 

\begin{defn}
\label{defn:eq_wl_pairs}
Two weak-liftable pairs $\wlp$ and $(\D_a',(D_{\overline{G}'}, \Pi_{\overline{G}'}))$ are said to be \textit{equivalent} if the following are satisfied:
\begin{enumerate}[(i)]
\item $\D_a$ and $\D_a'$ are equivalent with respect to $\pi$ (as in Definition \ref{defn:eq_data_sets}).
\item $D_{\overline{G}^{\ell}}=D_{\overline{G}'}$ with $(\ell,|\overline{G}'|)=1$.
\item $\Pi_{\overline{G}}=\pi^{-1} \circ \Pi_{\overline{G}'} \circ \pi$.
\end{enumerate}
\end{defn}
\begin{prop}\label{prop:wls}There exists a well-defined surjective map:
\begin{center}
$\left\{\parbox{38mm}{\raggedright Equivalence classes of $\E$-data sets}\right\}\xrightarrow{\makebox[1.5cm]{$\Psi$}}
\left\{\parbox{38mm}{\raggedright Equivalence classes of \\weak-liftable pairs} \right\}.$
\end{center}
\end{prop}
\begin{proof}
Let $\D_e=\bigl((n,m,i),g_0; [(\sigma_1,\bar{x}_1);m_1,t_1], \dots, [(\sigma_r,\bar{x}_r);m_r,t_r]\bigr)$ be an $\E$-data set of genus $g$. Let $H<\homeo(S_g)$ be a representative $\E$-action of $\D_e$. Then, there exists a unique $H_a\triangleleft H$ such that $H_a\cong A_n$. Thus, we have an $A_n$-action on $S_g$ and a $\mathbb{Z}_m$-action on $S_g/H_a$ induced by $H/H_a~(=\langle \overline{G} \rangle)$, as depicted below:
\begin{center}
     \begin{tikzcd}  
S_g
\arrow[rr, bend right, "/H~(\cong \E)" ']
\arrow[r, "{/H_a~(\cong A_n)}"] &[3em] \mathcal{O}_{H_a} \arrow[r, "/\langle \overline{G} \rangle~(\cong\Z_m)"] &[3em] \mathcal{O}_{H}
     \end{tikzcd}
     \end{center}
By applying Corollary \ref{cor:alt_corr}, we can obtain an alternating data set $\D_a$ corresponding to $H_a$-action on $S_g$. For convenience, we will now assume the existence of $\D_a$, which we will compute explicitly later in the proof. Let $\Pi_{\overline{G}} \in \Sigma_r$ be the permutation induced by $\overline{G} \in \map(\Orb_{H_a})$ and $D_{\overline{G}}$ denote the cyclic data set of the $\Z_m$-action induced by $\overline{G}$ on $\Orb_{H_a} \approx S_{g_0'}$. Define $\psi([\D_e]):= [\wlp]$. Let $H'$ be another $\E$-action weakly conjugate to $H$. Let $\phi_{H}:\Gamma\to \E$ and $\phi_{H'}:\Gamma' \to \E$ be surface kernel maps associated to $\D_e$ and $\D_e'$ , respectively. Then
\begin{center}\label{fig:wkconjugacy3}
\begin{minipage}{.28\textwidth}
\begin{tikzcd}[sep=12mm,
arrow style=tikz,
arrows=semithick,
diagrams={>={Straight Barb}}
                ]
\Gamma \dar["\phi_{H}" '] \rar["\psi",""name=U]
        & \Gamma' \dar["\phi_{H'}"]      \\
\E\rar["\chi"  ,""name=D] & \E
                     \ar[to path={(U) node[pos=.5,C] (D)}]{}
    \end{tikzcd} 
\end{minipage}
\begin{minipage}{.27\textwidth}
restricts to 
\end{minipage}
\hspace{-1.5cm}
\begin{minipage}{.25\textwidth}
\begin{tikzcd}[sep=12mm,
arrow style=tikz,
arrows=semithick,
diagrams={>={Straight Barb}}
                ]
\Gamma_1 \dar["\phi_{H_a}" '] \rar["\psi|_{\Gamma_1}",""name=U]
        & \Gamma_1' \dar["\phi_{H_a'}"]      \\
A_n\rar["\tilde{\chi}"  ,""name=D] & A_n
                     \ar[to path={(U) node[pos=.5,C] (D)}]{},
    \end{tikzcd}
\end{minipage}
\end{center} 
where  $\Gamma_1=\phi_{H}^{-1}(A_n)$, $\Gamma_1'=\phi_{H'}^{-1}(A_n)$, $\phi_{H_a}=\phi_{H}\big|_{\Gamma_1}$, $\phi_{H_a'}=\phi_{H'}\big|_{\Gamma_1'}$, and $\tilde{\chi}=\alpha$ if $\chi=\begin{pmatrix}
\alpha & \beta \\ 
0 & \delta
\end{pmatrix}.$ Thus, $H_a$ and $H_a'$ are also weakly conjugate, i.e. $\D_a$ and $\D_a'$ are equivalent.

We will now compute $\D_a$ and $D_{\overline{G}}$ explicitly. If $\Gamma_0=\phi_{H}^{-1}((id,\bar{0}))$, we see that $\Gamma_0,\Gamma_1 \triangleleft \Gamma$ and $\Gamma/\Gamma_1\cong (\Gamma/\Gamma_0)/(\Gamma_1/\Gamma_0)\cong \E/A_n \cong \Z_m.$ Hence, $\Gamma_1\triangleleft\Gamma$ is the unique index-$m$ subgroup with $\phi_{H}(\Gamma_1)=A_n$, and the corresponding surface kernel map for the $H_a$-action is given by $\phi_{H_a}:\Gamma_1\to A_n$. Let $\Gamma$ be as in Equation (\ref{eqn:eqn1}), then $\phi_{H}(\xi_j)=(\sigma_j,\bar{x}_j)$ for $1\leq j\leq r$, as given in Proposition \ref{prop:wlp}. It is well known that if $\gamma$ is an elliptic generator of $\Gamma_1$, then $|\gamma|= c_j/t_j$, where $c_j=|\xi_j|$ and $t_j=\circ(\xi_j\Gamma_1)$ in $\Gamma/\Gamma_1$, for some unique $j$. Furthermore, the conjugacy class of $\gamma$ in $\Gamma$ splits into $[\Gamma:\Gamma_1]/t_j$ many conjugacy classes in $\Gamma_1$. Here $c_j=|\xi_j|= |(\sigma_j,\bar{x}_j)|=m_j$, and we have: 
$$t_j=|\xi_j\Gamma_1| = |(\sigma_j,\bar{x}_j)A_n| =|\bar{x}_j|,$$ where $[\Gamma:\Gamma_1]=m.$ Thus, whenever $\bar{x}_j \neq \bar{0}$ and $m_j \neq t_j$, each $((\sigma_j,\bar{x}_j);m_j,t_j)$ in $\D_e$ contributes an even permutation in $\D_a$ that lies in the conjugacy class of $(\sigma_j,\bar{x}_j)^{t_j}$ in $A_n$. Moreover, when $\bar{x}_j=\bar{0}$ and the conjugacy classes of $\sigma_j$ does not split in $A_n$, each $((\sigma_j,\bar{0}_j);m_j,1)$ in $\D_e$ contributes $m$ even permutations in $\D_a$ that lie in the unique conjugacy class of $\sigma_j$ in $A_n$.  Futhermore, when $\bar{x}_j=\bar{0}$ and the conjugacy classes of $\sigma_j$ splits in $A_n$, each $((\sigma_j,\bar{0}_j);m_j,1)$ in $\D_e$ contributes $m$ even permutations in $\D_a$ such that $\frac{m}{2}$ permutations lies in the conjugacy class of $\sigma_j$ in $A_n$ but the other $\frac{m}{2}$ permutations lie in the other conjugacy class uniquely determined by the cycle type of $\sigma_j$. Thus, $\D_a$ is determined up to equivalence as in Definition \ref{defn:eq_data_sets}. Finally, by a straightforward computation we obtain that $D_{\overline{G}}=(m,g_0;(c_1,t_1),\cdots,(c_r,t_r)),$ where $c_j$'s are given by the equation $x_j=\frac{m}{t_j}\cdot c_j$, for $1\leq j \leq r$, and $(c_j,t_j)$ is omitted if some $t_j=1$.
\end{proof}

\noindent An immediate consequence of Proposition \ref{prop:wls} is the following equivalent version of our main result.
\begin{theorem}[Main Theorem 2]\label{thm:main2}
 $\wlp$ forms a weak-liftable pair if and only if there exists an $\E$-data set $\D_e$ such that $\Psi([\D_e])= [\wlp]$.
\end{theorem}

Note that the weak conjugacy class of free $A_n$-actions on $S_g$ is represented by the unique alternating data set
$\D_a=(n, g_0; -)$ for some $g_0\geq 2$. This implies $\frac{2-2g}{n!/2}=2-2g_0$, leading to $g= 1+ \frac{n!}{2}(g_0-1)$. The following corollary describes when such an action can be extended.
\begin{cor}\label{cor:cor1}
	For each $n,m,i$ with $n\geq5$ but $n\neq6$, $m\geq 2$, and $i\in \{0,1\}$, there exists a free $A_n$-action on $S_g$ having quotient orbifold genus $g_0$ that admits an extension to a free $\E$-action on $S_g$ if and only if $m$ divides $g_0-1$. Furthermore, there exists a free $A_n$-action on $S_g$ that extends to a non-free $\E$-action on $S_g$ if there exists a $\Z_m$-action on $S_{g_0}$ with quotient orbifold genus $\tilde{g_0}\geq2$.
\end{cor}
\begin{proof}Assume that there exists a free $A_n$-action on $S_g$ which can be extended to a free $\E$-action. Then, neccessarily, there exists a free $\Z_m$-action on $S_{g_0}$. Hence, by the Riemann-Hurwitz equation, $m$ divides $g_0-1$. Conversely, if $m$ divides $g_0-1$ then a representative free alternating action of $\D_a=(n, g_0; -)$ can be extended to a free $\E$-action since $\Psi([\D_e])=[\wlp]$, where $\D_e=((n,m,i),\tilde{g_0}; -)$ with $\tilde{g_0}=1+(g_0-1)/m$, $D_{\overline{G}}=(2,\tilde{g_0}; -)$, and $\Pi_{\overline{G}}= id$.
	
	To prove the second part, assume that there exists a $\Z_m$-action on $S_{g_0}$ whose conjugacy class is represented by $D_{\overline{G}'}=(m,\tilde{g_0}; (c_1,m_1), \dots, (c_r,m_r))$ with $\tilde{g_0}\geq2$. Then, we verify that $\D_e'=((n,m,i),\tilde{g_0};[(id,\overline{mc_1/m_1});m_1,m_1],\dots,[(id,\overline{mc_r/m_r});m_r,m_r])$ satisfy the hypothesis of an $\E$-data set. Furthermore, $\Psi([\D_e])=(\D_a',(D_{\overline{G'}}, \Pi_{\overline{G'}}))$. Hence, the assertion follows.
\end{proof}

\begin{exmp} Consider an $A_7\times\Z_{30}$-data set 
	$$\D_e=\bigl((7,30,0), 2; [(\sigma_1,\bar{x}_1);m_1,t_1],\dots, [(\sigma_5,\bar{x}_5);m_5,t_5]\bigr),$$
	where $(\sigma_1,\bar{x}_1)=((1~2~3),-\bar{1})$, $(\sigma_2,\bar{x}_2)=((1~2~3~4~5),-\bar{3})$, $(\sigma_3,\bar{x}_3)=((1~2)(3~4),-\bar{1})$, $(\sigma_4,\bar{x}_4)=((1~2~3)(4~5~6),\bar{0})$, and $(\sigma_5,\bar{x}_5)=(id,\bar{5})$.
	
To determine $\D_a$ and $D_{\overline{G}}$ corresponding to $\D_e$, we first compute that: $(m_1,t_1)=(30,30)$, $(m_2,t_2)=(10,10)$, $(m_3,t_3)=(30,30)$, $(m_4,t_4)=(3,1)$, and $(m_5,t_5)=(6,6)$. Since $m_j/t_j=1$ for $j=1,2,3,$ and $5$, they contribute nothing in $\D_a$. For $j=4$, $m_4/t_4=3$, hence it contributes $m/t_4=30$ cone points in $\D_a$, conjugate to $(1~2~3)(4~5~6)$ in $A_7$, of orders $3$. Thus, up to equivalence, $$\D_a=\bigl(7,86; [(1~2~3)(4~5~6);3]^{[30]}\bigr).$$ Furthermore, $D_{\overline{G}}=(30,2;(c_1,30),(c_2,10),(c_3,30),(c_5,6))$ where $c_i$'s are given by the equation: $c_i=\frac{t_i}{m}\cdot x_i$. Thus, $D_{\overline{G}}=(30,2;(29,30),(9,10),(29,30),(1,6))$. Finally, it is straightforward to see that $\Pi_{\overline{G}}=(1~2~\dots~30)$ is a $30$-cycle. Thus $\Psi([\D_e])= [\wlp]$.
\end{exmp}

\begin{exmp} Consider an $A_7\rtimes_{\varphi}\Z_{30}$-data set 
	$$\D_e=\bigl((7,30,1), 2; [(\sigma_1,\bar{x}_1);m_1',t_1'],\dots, [(\sigma_5,\bar{x}_5);m_5',t_5']\bigr),$$
where $(\sigma_1,\bar{x}_1)=((1~2~3),-\bar{1})$, $(\sigma_2,\bar{x}_2)=((1~2~3~4~5),-\bar{3})$, $(\sigma_3,\bar{x}_3)=((1~2)(3~4),-\bar{1})$, $(\sigma_4,\bar{x}_4)=((1~2~3)(4~5~6),\bar{0})$, and $(\sigma_5,\bar{x}_5)=(id,\bar{5})$. 

To determine $\D_a$ and $D_{\overline{G}}$ corresponding to $\D_e$, we compute that : $(m_1',t_1')=(30,30)$, $(m_2',t_2')=(20,10)$, $(m_3',t_3')=(30,30)$, $(m_4',t_4')=(3,1)$, and $(m_5',t_5')=(6,6)$. Since $m_j'/t_j'=1$ for $j=1,3,$ and $5$, they contribute nothing in $\D_a$. For $j=2$, $m_2'/t_2'=2$, hence it contributes $m/t_2'=3$ cone points in $\D_a$, conjugate to $(1~4)(3~5)$ in $A_7$, of orders $2$. For $j=4$, $m_4'/t_4'=3$, hence it contributes $m/t_4'=30$ cone points in $\D_a$, conjugate to $(1~2~3)(4~5~6)$ in $A_7$, of orders $3$. Thus, up to equivalence, $$\D_a=\bigl(7,86; [(1~4)(3~5);2]^{3},[(1~2~3)(4~5~6);3]^{30}\bigr).$$ Furthermore, $D_{\overline{G}}=(30,2;(29,30),(9,10),(29,30),(1,6))$, remains unchanged. Finally, it is straightforward to see that $\Pi_{\overline{G}}=(1~2~3)(4~5~\dots~33)$.
\end{exmp}

\section{Applications} 
\label{sec:apps}

\subsection{Lifting involutions under alternating covers}
\label{subsec:liftinv}
Consider an $H$-action on $S_g$, where $H \cong A_n$, and the induced branched cover $p: S_g \to \Orb_H= S_g/H$. For a liftable involution $\overline{G} \in \map(\Orb_H)$, the short exact sequence (\ref{eqn:eqn4}) restricts to
   \[1 \to \mathrm{Deck}(p) \to \widetilde{H} \to \langle \overline{G} \rangle \to 1,\tag{6} \label{eqn:eqn6}
        \]
where $\widetilde{H}$ is either isomorphic to $A_n \times \Z_2$ or $\Sigma_n$. To begin with, we determine when $\widetilde{H}$ is a symmetric group. For this purpose, we first define the following.

\begin{defn}
\label{defn:wls}
A weak-liftable pair $\wlp$ of degree $(n,2)$ is called a \textit{weak-liftable symmetric pair} (abbreviated as \textit{WLS-pair}) of degree $n$ if there exists a representative $H<\map(S_g)$ of $[\D_a]$ such that the following conditions hold. 
\begin{enumerate}[(i)] 
\item $\overline{G}$ lifts to a $G$ under the alternating cover $S_g \to S_g/H$.
\item $\langle H\cup\{G\} \rangle \cong \Sigma_n$.
\end{enumerate}
\end{defn}

It is implicit from the definition of a weak-liftable pair $\wlp$ of degree $(n,2)$ that $|\Pi_{\overline{G}}|$ divides $|\overline{G}|=2$. Further, the equivalence in Definition \ref{defn:eq_wl_pairs} can be adapted for the symmetric case as follows.

\noindent An immediate consequence of Proposition \ref{prop:wls} is the following.
\begin{theorem}\label{thm:wls}
For $n\geq5$ but $n\neq6$, $\wlp$ forms a WLS-pair of degree $n$ if and only if there exists a symmetric data set $\D_s$ of degree $n$ such that $\Psi([\D_s])= [\wlp]$.
\end{theorem}
The Corollary~\ref{cor:cor1} can be adapted for the symmetric case as follows.
\begin{cor}\label{cor:cor3}
For each $n\geq5$ but $n\neq6$, there exists a free $A_n$-action on $S_g$ having quotient orbifold genus $g_0$ that  admits an extension to a free $\Sigma_n$-action on $S_g$ if and only if $g_0\geq 3$ is odd. Furthermore, there exists a free $A_n$-action on $S_g$ that extends to a non-free $\Sigma_n$-action on $S_g$ if $g_0$ is even.
\end{cor}

The following is another criterion for weak liftability, derived from \cite[Theorem 3.2.]{broughton_normalizer}. 

\begin{prop}\label{prop:main4}
Let $\D_a$ be an alternating data set and $D_{\overline{G}}$ be a cyclic data set. Then $\wlp$ forms a weak-liftable pair if and only if there exists $H \in [\D_a]$ and a $\chi_\omega\in \mathrm{Aut}(A_n)$ such that $\phi_H \circ \overline{G}_{\ast} = \chi_\omega \circ \phi_H$. 
\end{prop}

Let $\D_a=\bigl(n,g_0; [\sigma_1;m_1], \dots, [\sigma_r;m_r]\bigr)$, let $\{k_{1j},\ldots,k_{\ell_jj}\}$ be the multiset of the lengths of distinct cycles in the canonical decomposition of $\sigma_i$ into disjoint cycles, and let 
$$\Lambda = \left\{j \in \{1,2, \dots, r\} \, \Biggl\vert\,\begin{matrix*}[l]
k_{1j},\dots , k_{\ell_jj}~\text{are distinct odd}\\
\text{integers and}~\displaystyle\sum_{p=1}^{\ell_j}\,k_{pj}\,\geq n-1
\end{matrix*}
\right\}.$$ The following proposition gives a necessary condition of forming a weak liftable pair, follows directly from Theorem \ref{thm:zieschang} and Proposition \ref{prop:main4}.
\begin{prop}\label{prop:admissible_permutation}
Let $\wlp$ form a weak-liftable pair. Then
\begin{enumerate}[(i)]
\item $\sigma_j$ and $\sigma_{\Pi_{\overline{G}}(j)}$ have the same cycle type, and

\item exactly one of the following conditions holds:
    \begin{enumerate}[(a)]
    \item $\sigma_j$ is conjugate to $\sigma_{\Pi_{\overline{G}}(j)}$ in $A_n$ for all $j\in \Lambda$.
    \item $\sigma_j$ is not conjugate to $\sigma_{\Pi_{\overline{G}}(j)}$ in $A_n$ for all $j\in \Lambda$.
    \end{enumerate}
\end{enumerate}
\end{prop}
\noindent As a consequence of Proposition \ref{prop:admissible_permutation}, we have the following.
\begin{cor}\label{cor:self_normalizing}
	Let $\D_a=\bigl(n,g_0; [\sigma_1;m_1], \dots, [\sigma_r;m_r]\bigr)$ such that $\sigma_i$ is not conjugate to $\sigma_j$ in $\Sigma_n$ for $1\leq i,j\leq r$. Then any representative $H<\map(S_g)$ of $[\D_a]$ does not extend to any $\E$-action. Moreover, any representative $H<\map(S_g)$ of $[\D_a]$ is self-normalizing in $\map(S_g)$ whenever $r=3$.
\end{cor}

%

\noindent Now we will show that any involution on $S_{g_0}$ having quotient orbifold genus $\tilde{g}_0\geq 1$, weakly lifts under a free alternating cover: $S_g \overset{/A_n}{\longrightarrow} S_{g_0}$.
\begin{cor}
Let $\D_a=(n,g_0;-)$ be of genus $g$ and $D_{\overline{G}}=(2,\tilde{g}_0;(1,2),\overset{\ell}{\cdots},(1,2))$ be of genus $g_0$ with $\Pi_{\overline{G}}=id$ and $\tilde{g}_0 \geq 1$. Then $\wlp$ forms a WLS-pair, and hence a weak-liftable pair.
\end{cor}
\begin{proof}
Consider the tuple $((n,2,1),\tilde{g}_0;[(id,\bar{1});2,2]^{[\ell]})$. Then by of a cyclic data set, $\ell$ must be even. Hence, $\D_s:=((n,2,1),\tilde{g}_0;[(id,\bar{1});2,2]^{[\ell]})$ satisfies the conditions of a symmetric data set since $\prod_{j = 1}^{\ell} (id,\bar{1})$ is identity, which is an even permutation. Finally, it follows from Proposition \ref{prop:wls} that $\Psi([\D_s])=[\wlp]$.
\end{proof}

\noindent Thus, the following result is obtained directly from Theorem~\ref{thm:wls} and Proposition~\ref{prop:extension}.
\begin{theorem}\label{thm:alternating_extensions}
For $n\geq5$ but $n\neq 6$, consider an $H<\map(S_g)$ such that $H\cong A_n$ and let $\D_a$ be an alternating data set representing the weak conjugacy class of the $H$-action. Then there exists an $H'<\map(S_g)$ that is weakly conjugate to $H$ and an $H''<\map(S_g)$ with $[H'':H']=2$ if and only if  $\wlp$ is a weak-liftable pair.  Furthermore:
\begin{enumerate}[(i)]
\item $H''$ can be chosen to be isomorphic to $\Sigma_n$ if and only if  $\wlp$ forms a WLS-pair, and
\item $H''$ can be chosen to be isomorphic to $A_n \times \Z_2$ if $\wlp$ is a weak-liftable pair that is not WLS-pair.
\end{enumerate} 
\end{theorem}

\subsubsection{Examples of liftable involutions}
In this section, we explore the liftability of involutions under the alternating covers induced by the alternating actions on $S_g$.

\begin{exmp}[Icosahedral action]
\label{eg:icosa}
Consider the $A_5$-action on $S_{19}$ whose weak conjugacy class is represented by 
\begin{center}
$\D_a=(5,0;[(1~2)(3~4);2]^{[2]}, [(1~5~4~3~2);5], [(1~2~3~4~5);5])$.
\end{center}
\noindent Any periodic mapping class in $\map(\Orb_H)$ for $H \in \D_a$ has a unique cyclic data set $D=(2,0;(1,2)^{[2]})$. By Proposition~\ref{prop:admissible_permutation}, any liftable involution can only induce the permutations $\Pi_{\overline{G}_1}=(1~2)$, $\Pi_{\overline{G}_2}=(3~4)$, and $\Pi_{\overline{G}_3}=(1~2)(3~4)$ in $\Sigma_4$. The quotient orbifolds for these permutations have the signatures $(0;2,10,10)$, $(0;4,4,5)$, and $(0;2,2,2,5)$, respectively. By Definition~\ref{defn:alt_sym_data_set}, the signatures $(0;4,4,5)$ and $(0;2,2,2,5)$ give rise to unique symmetric data sets, up to equivalence, represented by
\begin{gather*}
\D_s = ((5,2,1),0;[((1~4)(2~5),\bar{1});4,2], [((2~1~4~3~5),\bar{1});4,2], [((1~2~3~4~5),\bar{0});5,1]) \\ \text{and} \\
\D_s' = ((5,2,1),0;[((1~2~4),\bar{1});2,2], [((1~2~5),\bar{1});2,2], [((1~3)(2~4),\bar{0});2,1], [((1~2~4~5~3),\bar{0});5,1]), 
\end{gather*}
respectively. Consequently,
$\Psi([\D_s])=[(\D_a,(D_{\overline{G}_2}, \Pi_{\overline{G}_2}))]$, and $\Psi([\D_s'])= [(\D_a,(D_{\overline{G}_3}, \Pi_{\overline{G}_3}))]$ with $D_{\overline{G}_2}=D_{\overline{G}_3}=D$. The two $\Z_2$-actions $D_{\overline{G}_2}$ and $D_{\overline{G}_3}$ on $\Orb_H$ are shown in Figure \ref{fig:fig4}, where labeled pair $(z,n)$ represents a cone point $z$ along with its order $n$. The signature $(0;2,10,10)$ does not correspond to a symmetric data set of degree $5$ since no element in $\Sigma_5$ has order $10$. However, it can be verified that the signature $(0;2,10,10)$ corresponds to an $E_{4,2}^0$-data set, resulting in an $A_4\times \Z_2$-subgroup of $\map(S_5)$. In conclusion, $[(\D_a,(D_{\overline{G}_2}, \Pi_{\overline{G}_2}))]$ and $[(\D_a,(D_{\overline{G}_3}, \Pi_{\overline{G}_3}))]$ only form WLS-pair whereas $[(\D_a,(D_{\overline{G}_1}, \Pi_{\overline{G}_1}))]$ forms a weak-liftable pair that is not a WLS-pair.
\begin{figure}[ht]
\centering
\parbox{5cm}{

\labellist
\small
\pinlabel $(x_1,2)$ at 248 410
\pinlabel $\pi$ at 210 480
\pinlabel $\G_2$ at 120 480
\pinlabel $(x_2,2)$ at 248 40
\pinlabel $(x_3,5)$ at -65 250
\pinlabel $(x_4,5)$ at 415 250
\endlabellist
\centering
   \includegraphics[scale=.23]{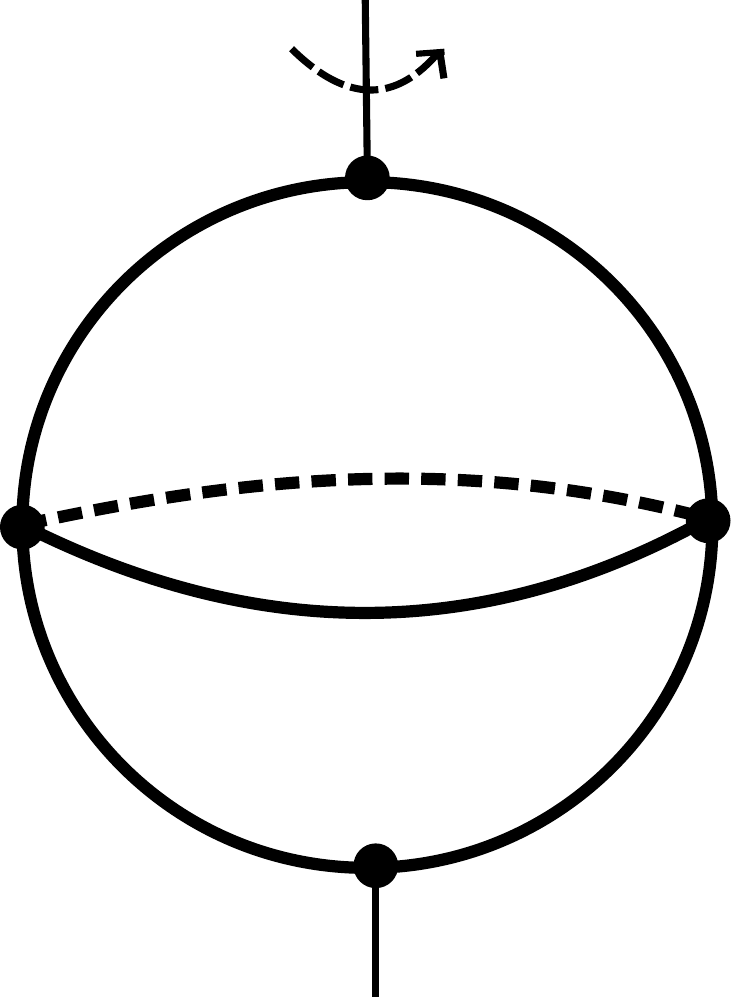}}
\qquad
\begin{minipage}{5cm}
\labellist
\small
\pinlabel $(y_1,2)$ at 242 280
\pinlabel $(y_2,2)$ at 178 152
\pinlabel $\pi$ at 210 480
\pinlabel $\G_3$ at 120 480
\pinlabel $(y_3,5)$ at -65 210
\pinlabel $(y_4,5)$ at 410 210
\endlabellist
\centering
   \includegraphics[scale=.23]{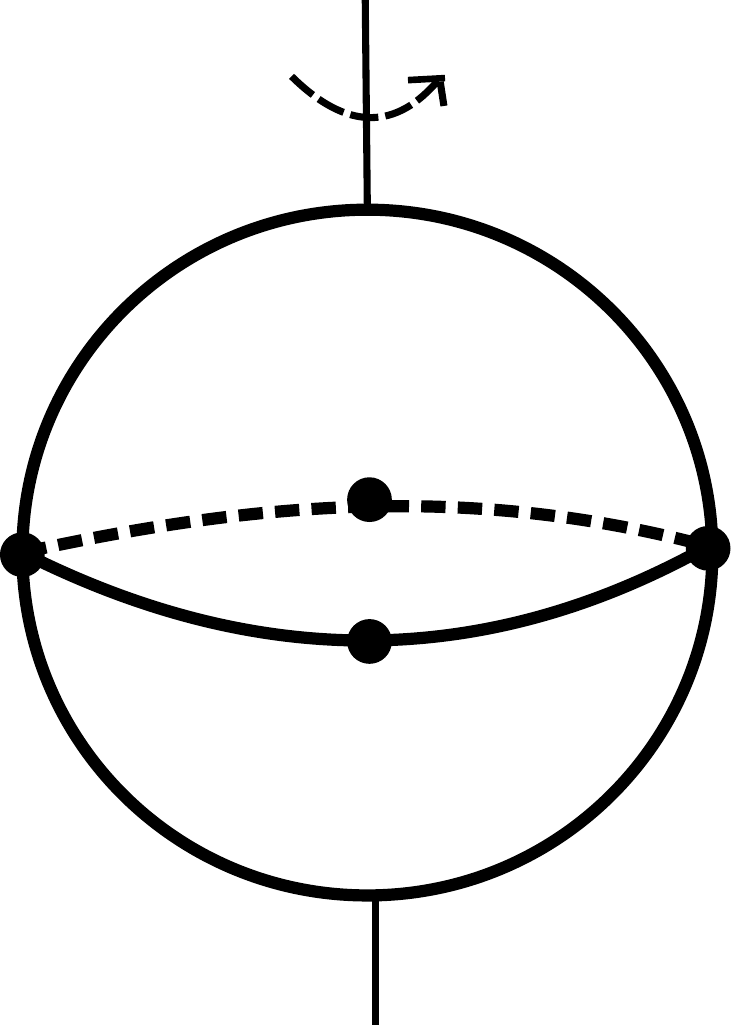}
\end{minipage}
\caption{The two $\Z_2$-actions on $\Orb_H$ induced by $D_{\overline{G}_2}$(left) and $D_{\overline{G}_3}$(right).}
\label{fig:fig4}
\end{figure}
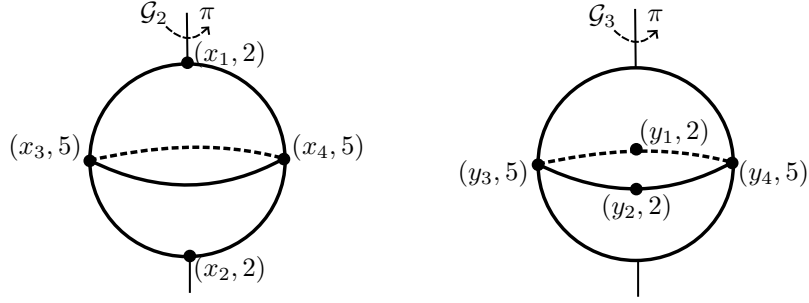
\end{exmp}

\begin{exmp}[Dodecahedral action]
\label{exmp:lift_dodeca}
Consider the $A_5$-action on $S_{11}$ represented by $\D_a=(5,0;[(1 3)(2 4);2]^{[2]},[(3 5 4);3],[(3 4 5);3])$. By Proposition~\ref{prop:admissible_permutation}, any liftable involution can only induce the permutations $\Pi_{\overline{G}_1}=(1~2)$, $\Pi_{\overline{G}_2}=(3~4)$, and $\Pi_{\overline{G}_3}=(1~2)(3~4)$ in $\Sigma_4$. The quotient orbifolds for these permutations have signatures $(0;2,6,6)$, $(0;3,4,4)$, and $(0;2,2,2,3)$, respectively. By Definition~\ref{defn:alt_sym_data_set}, the signatures $(0;2,6,6)$ and $(0;3,4,4)$ give rise to unique symmetric data sets, up to equivalence, represented by
\begin{gather*}
\D_s = ((5,2,1),0;[((2~4~1~3~5),\bar{1});6,2],[((3 ~4~5),\bar{1});6,2],[((1~3)(2~4),\bar{0});2,1]) \\ \text{and} \\ \D_s' = ((5,2,1),0;[((2~4~3),\bar{1});4,2],[((1~3~5),\bar{1});4,2],[((3~4~5),\bar{0});3]),
\end{gather*}
respectively. Consequently,
$\Psi([\D_s])=[(\D_a,(D_{\overline{G}_1}, \Pi_{\overline{G}_1}))]$ and $\Psi([\D_s'])= [(\D_a,(D_{\overline{G}_2}, \Pi_{\overline{G}_2}))]$ where $D_{\overline{G}_1}=D_{\overline{G}_2}=(2,0;(1,2)^{[2]})$. The signature $(0;2,2,2,3)$ only corresponds to an $E_{4,2}^0$-data set, resulting an $A_4\times \Z_2$-subgroup of $\map(S_{11})$. Therefore, $[(\D_a,(D_{\overline{G}_3}, \Pi_{\overline{G}_3}))]$ forms a weak-liftable pair that is not a WLS-pair.
\end{exmp}

\subsection{Embeddable cyclic actions inside $\E$-action}
\label{subsec:embcyc}
In this section, we show that an irreducible periodic mapping class is not contained within an $\E$-subgroup of $\map(S_g)$. Additionally, we show that certain $\E$-subgroups of $\map(S_g)$ do not contain the hyperelliptic involution. Let $\ell(n)$ denote the largest order of an element in $\Sigma_n$. The following proposition provides an asymptotic bound for the order of an element in such a subgroup.

\begin{prop}\label{prop:bound}
Let $H\leq\map(S_g)$ be isomorphic to $\E$ with $g\geq2, m \geq 2,$ and $n\geq5$ but $n\neq 6$. Given $k\in \mathbb{N}$, for any $F\in H$, we have $|F|\leq g/k$ whenever $n \geq n_0$, where $n_0\in \mathbb{N}$ is the smallest integer satisfying $\frac{1}{4}\cdot\frac{n_0!}{\ell(n_0)} > 84k$. In particular, for any $F\in H$, we have $|F|\leq g$ whenever $n \geq 7$.
\end{prop}

\begin{proof}
It is known~\cite{Jean-Pierre} that $\ell(n) < e^{n/e}$. The sequence $f(n):=\frac{n!}{e^{n/e}}$ diverges to $+\infty$, and since $\frac{n!}{\ell(n)}>\frac{n!}{e^{n/e}}=f(n)$, the sequence $\Bigl\{\frac{n!}{\ell(n)}\Bigr\}$ also diverges to $+\infty$. Futhermore, any element $(\sigma,\bar{x}) \in \E$ has its order bounded by $\ell(n) \cdot 2m$.

Suppose, on the contrary, that for a fixed $k$, there exists a monotonically diverging sequence of natural numbers $\{n_j\}$ and corresponding elements $F_{n_j} \in H_{n_j}\leq \mathrm{Mod}(S_{g_{n_j}})$ such that $|F_{n_j}|>g_{n_j}/k$. Since $|H_{n_j}|\leq 84(g_{n_j}-1)$, we have: 
    $$\frac{1}{4}\cdot\frac{n_j!}{\ell(n_j)}=\frac{n_j!/2\cdot m}{\ell(n_j) \cdot 2m}\leq \frac{|H_{n_j}|}{|F_{n_j}|} < \frac{84(g_{n_j}-1)}{g_{n_j}/k} <84k,$$
which contradicts the fact that the subsequence $\Bigl\{\frac{n_j!}{\ell(n_j)}\Bigr\}$ of $\Bigl\{\frac{n!}{\ell(n)}\Bigr\}$ diverges to $+\infty$. Therefore the assertion holds.

Moreover, for $k=1$, the integer $n_0$ can be taken as $7$ since $\frac{1}{4}\cdot\frac{7!}{\ell(7)}=105>84\cdot1$. Therefore, for any $F\in H$, we have $|F|\leq g$ whenever $n \geq 7$.

\end{proof}

An \textit{irreducible periodic mapping class} is characterized by its corresponding quotient orbifold, which is a sphere with three cone points (see~\cite{gilman}). Using this characterization, the following corollary is derived as an application of the above proposition.

\begin{cor}
\label{prop:noirr}
For $n\geq7$, if $H\leq \mathrm{Mod}(S_g)$ is such that $H\cong \E$, then $H$ contains no irreducible mapping class.
\end{cor}
\begin{proof}
Let $F \in \mathrm{Mod}(S_g)$ be an irreducible periodic mapping class, and let $\sig(\Orb_F)=(0;n_1,n_2,n_3)$. From the Riemann-Hurwitz equation, we have:
$$\begin{array}{lrl}
& \frac{2-2g}{|F|} &=~2-(1-\frac{1}{n_1})-(1-\frac{1}{n_2})-(1-\frac{1}{n_3}).\\
\Rightarrow & \frac{2-2g}{|F|} &=~-1+(\frac{1}{n_1}+\frac{1}{n_2}+\frac{1}{n_3})\geq~-1+\frac{3}{|F|}.\\
\Rightarrow & 1 &\geq~\frac{2g-2}{|F|}+\frac{3}{|F|}=\frac{2g+1}{|F|}.\\
\Rightarrow & |F| &\geq~2g+1.
\end{array}$$
Thus, $|F|> 2g$. But, by Proposition~\ref{prop:bound}, any mapping class in $H\cong \E$ has an order less than $g$ for $n \geq 7$. Hence $F\notin H$ for $n \geq 7$.
\end{proof}

\begin{prop}
	Let $H\leq \mathrm{Mod}(S_g)$ be such that $H\cong \E$, where $n\geq 10$, $m$ is odd if $i = 0$, and $\frac{m}{2}$ is odd if $i = 1$. Then $H$ does not contain a hyperelliptic involution.
\end{prop}
\begin{proof}
	Suppose the weak conjugacy class of $H$ is represented by a data set $\D_e$ as in Definition~\ref{defn:H_data_set}. Since a hyperelliptic involution has a quotient orbifold of genus zero, it is evident that if the orbifold $S_g/H$ has a positive genus, no hyperelliptic involution can exist inside $H$. Therefore, we assume $g_0=0$. By the Riemann-Hurwitz equation, we have
	$$2+2g=4+\frac{n!}{2}m \biggl(r-2-\sum_{j=1}^{r}\frac{1}{m_j}\biggr).$$
	Let $X$ represents the expression on the right-hand side of the preceding equation. Denoting the hyperelliptic involution by $\sigma\in \mathrm{Mod}(S_g)$, we have $D_{\sigma}=\left(2,0; (1,2)^{[2g+2]}\right)$. Now, assume $$Y:=|\mathbb{F}_{\sigma}(1,2)|=|C_{\E}(\sigma)| \cdot \sum\limits_{\substack{1\leq j\leq r, ~2\mid m_j \\ \sigma\sim_{\E} \sigma_j^{m_j/2}}}\frac{1}{m_j},$$ in view of Proposition~\ref{prop:cyclicgen}. To rule out the possibility of $g_0=0$, it suffices to establish that $X>Y$ for $n\geq 10$, since both expressions compute the number of fixed points of $\sigma$. We break the argument into the following three cases.
	
	\noindent\textit{Case (i): $r\geq 5$.} In this case, we have:
	\begin{align*}
	X 
	&= 4+\frac{n!}{2}m \biggl(r-2-\sum_{j=1}^{r}\frac{1}{m_j}\biggr) &\\
	&\geq \frac{n!}{2}m \biggl(\frac{r}{2} -2\biggr) \quad \biggl( \text{as } \frac{1}{m_j} \leq \frac{1}{2} \text{ for all } 1 \leq j \leq r \biggr) & \\
	&\geq \left(n-2\right)!m \frac{r}{2} \quad \biggl( \text{as } n \left( n-1 \right) \geq 10 \geq 2 \left( 1 + \frac{4}{r-4} \right) \text{ for } n \geq 4, r \geq 5 \biggr)& \\
	&\geq |C_{\E}(\sigma)| \cdot \sum\limits_{\substack{1\leq j\leq r, ~2\mid m_j \\ \sigma\sim_{\E} \sigma_j^{m_j/2}}}\frac{1}{m_j} \quad \biggl( \text{as } \frac{1}{m_j} \leq \frac{1}{2} \text{ for all } 1 \leq j \leq r \text{ and } |C_{\E}(\sigma)| \leq (n-2)!m \biggr).
	\end{align*}
	
	\noindent\textit{Case (ii): $r=4$.} In this case, we have:
	\begin{align*}
	X 
	&= 4+\frac{n!}{2}m \biggl(4-2-\sum_{j=1}^{4}\frac{1}{m_j}\biggr) &\\
	&\geq \frac{n!}{12}m \quad \biggl( \text{as } \sum_{j=1}^{4}\frac{1}{m_j} \leq \frac{1}{2} +\frac{1}{2}+\frac{1}{2}+\frac{1}{3} = \frac{11}{6} \biggr) & \\
	&\geq 2\left(n-2\right)!m \quad \biggl( \text{as } n \left( n-1 \right) \geq 24 \text{ for } n \geq 6 \biggr)& \\
	&\geq |C_{\E}(\sigma)| \cdot \sum\limits_{\substack{1\leq j\leq r, ~2\mid m_j \\ \sigma\sim_{\E} \sigma_j^{m_j/2}}}\frac{1}{m_j} \quad \biggl( \text{since } \sum_{j=1}^{4}\frac{1}{m_j} \leq \frac{11}{6}<2 \text{ and } |C_{\E}(\sigma)| \leq (n-2)!m \biggr).
	\end{align*}
	
	\noindent\textit{Case (iii): $r=3$.} In this case, we have:
	\begin{align*}
		X 
		&= 4+\frac{n!}{2}m \biggl(3-2-\sum_{j=1}^{3}\frac{1}{m_j}\biggr) &\\
		&\geq \frac{n!}{84}m \quad \biggl( \text{as } \sum_{j=1}^{3}\frac{1}{m_j} \leq \frac{1}{2} +\frac{1}{3}+\frac{1}{7} = \frac{41}{42} \biggr) & \\
		&\geq \left(n-2\right)!m \quad \biggl( \text{as } n \left( n-1 \right) \geq 84 \text{ for } n \geq 10 \biggr)& \\
		&\geq |C_{\E}(\sigma)| \cdot \sum\limits_{\substack{1\leq j\leq r, ~2\mid m_j \\ \sigma\sim_{\E} \sigma_j^{m_j/2}}}\frac{1}{m_j} \quad \biggl( \text{as } \sum_{j=1}^{3}\frac{1}{m_j} \leq \frac{41}{42}<1 \text{ and } |C_{\E}(\sigma)| \leq (n-2)!m \biggr).
	\end{align*}
Hence, our assertion follows.
\end{proof}

\subsection{Bound on cyclic action under alternating covers}
\label{subsec:boundcyc}
Let $E_{n,m}^i$ acts on $S_g$ with quotient orbifold $S_{g_0, \ell}$. In this subsection, we will see the bound on $m$ for some condition on order of $E_{n,m}^i$.

\begin{theorem}
	Let $H < \map(S_g)$ such that $H \cong E_{n,m}^i$.  If $|H| > 5g-5$,  then $m \leq 26$.
\end{theorem}

\begin{proof}
	Let $H \cong A_n \rtimes_{\phi} \langle G  \rangle$. Let $H$ acts on $S_g$ with a quotient orbifold $\Orb_H$ of genus $g_0$ with $\ell$ cone points. The $A_n$-action on $S_g$ induces a $\langle \bar{G} \rangle$-action on $\Orb_{A_n}$ with $|\bar{G}|=m$.  By the Riemann-Hurwitz equation (Theorem~\ref{thm:riemann}(iii)) and our assumption that $|H|>5g-5$, it follows that $g_0 = 0$ and $\ell \in \{3,4\}$.  The following cases arise. 
	
\begin{enumerate}[{Case} 1:]

\item $g_0 = 0$, $\ell = 4$ and $|H|> 5g-5$.  In this case, the possible signature for $\Orb_H$ are:
	$(0;2,2,2,k)$ for $3 \leq k \leq 9$ and $(0;2,2,3,3)$.  Furthermore, for these signatures, it can be verified from the existence of induced $\langle \bar{G} \rangle$- action and Theorem~\ref{thm:cyc_conj_class}, that the possible values of $m$ are $2,3,$ and $6$.

\item $g_0 = 0, \ell = 3$ and $|H|> 5g-5$.  In this case,  we have the following possible signatures for $\Orb_h$:
\begin{enumerate}[(a)]
\item $(0;2,3,k)$ for $k \geq 7$, 
\item $(0;2,4,k)$ for $k \geq 5$, 
\item $(0;2,j,k)$ for $5 \leq j \leq 19$, $j \leq k < \frac{10j}{j-10}$ $\left(\text{if } \frac{10j}{j-10}>0 \right)$, 
\item $(0;3,3,k)$ for $k \geq 4$, 
\item $(0;3,j,k)$ for $4 \leq j \leq 7, j \leq k < \frac{15j}{4j-15}$, 
\item $(0;4,4,k)$ for $4 \leq k < 10$,  and
\item $(0;4,5,k)$ for $5 \leq k < 7$. 
\end{enumerate}
For each of the above signatures, it can be shown from the existence of induced $\langle \bar{G} \rangle$- action and Theorem~\ref{thm:cyc_conj_class}, that $m \in \{2,3\ldots,18\} \cup\{19, 22,26\}$.	
\end{enumerate}
\end{proof}

\subsection{Lifting infinite order mapping classes under alternating covers}
\label{subsec:liftinf}
So far, we have dealt with the liftability of periodic mapping classes under alternating covers. In this section, we explore the liftability of some infinite order maps (e.g. Dehn twists) under alternating covers. This leads to the analysis of the subgroups of $\map(S_g)$ that are extensions of the infinite cyclic group by alternating groups, i.e., the groups, which we denote by $E_{n,\infty}$, fitting into the short exact sequence:
 \[1 \to A_n \to E_{n,\infty} \to \Z \to 1. \tag{$\star$}\]
It also follows from the theory of group extensions \cite[Chapter 11]{robinson} that any coupling $\Phi: \Z \to \out(A_n)$ is realized by some extensions of $\Z$ by $A_n$. Moreover, there is a unique equivalence class of extensions of $\Z$ by $A_n$ with a coupling $\Phi$. Hence, one can conclude that the short exact sequence ($\star$) always splits, and we have the following.
\begin{prop}\label{prop:inf_extension}For the short exact sequence: $1\to A_n \to E_{n,\infty} \to \Z \to 1$, exactly one of the following holds.
\begin{enumerate}[(i)]
\item  $E_{n,\infty}$ is isomorphic to $A_n \times \Z$.
\item  $E_{n,\infty}$ is isomorphic to $A_n \rtimes_{\phi} \Z$, where $\varphi:\Z \to \aut(A_n)$ is a homomorphism determined by $\varphi(1)=\chi_{(1~2)} \in \aut(A_n)$, defined as $\chi_{(1~2)}(\sigma)=(1~2) \sigma (1~2)$.
\end{enumerate}
\end{prop}
A \textit{multicurve} in $S_g$ is the union of a finite collection of disjoint non-isotopic
essential (not homotopic to a point) simple closed curves in $S_g$. By a \textit{multitwist associated with a multicurve $\mathcal{C}$} in $\map(S_g)$, we mean a finite product of Dehn
twists about curves in $\mathcal{C}$. To construct $E_{n,\infty}$ subgroups in $\map(S_g)$, we restrict ourselves, in this section, to the case where the $\Z$-factor corresponds to a multitwist.

\begin{lem}\label{lem:scc}
If the group $A_n \rtimes \Z$ embeds in $\map(S_g)$, then the $\Z$-factor of the embedding is not generated by a power of a Dehn twist.
\end{lem}
\begin{proof}
First, we consider the case of the direct product.  Suppose we assume on the contrary that $A_n \times \Z$ embeds in $\map(S_g)$ such that the embedded $\Z$-component is generated by a power of a Dehn twist $T_c ^\alpha$.  This leads to a subgroup $H\leq \map(S_g)$ such that $H \cong A_n$ and $T_c^\alpha$ with the property that, for all $F\in H$,
$$T_c^\alpha \cdot F = F\cdot T_c^\alpha \iff T_c^\alpha = F\cdot T_c^\alpha \cdot F^{-1}= T_{F(c)}^\alpha.$$
In other words, $H$ preserves a simple closed curve $c$. For the case $c$ nonseparating, we cut $S_g$ along $c$ and cap the two boundaries with once marked disks. Thus, the subgroup $H\leq \map(S_g)$ induces an isomorphic subgroup $H'< \map(S_{g-1,2})$. Let $\mathrm{PMod}(S_{g-1,2})$ be the mapping classes that fix each marked point (also called the \textit{pure mapping class group}). Since $H'$ is generated by two elements of odd order, $H'$ fixes the marked point in each disk (involved in the capping). Therefore, it follows that $H'< \mathrm{PMod}(S_{g-1,2})$ contradicting the fact that all finite subgroups of $\mathrm{PMod}(S_{g-1,2})$ are cyclic~\cite[Lemma 4.2]{breuer}. 

For the case when $c$ separating, by cutting and capping, we obtain two subsurfaces $S_{g_0,1}$ and $S_{g_0',1}$ with $g=g_0+g_0'$. Thus the subgroup $H\leq \map(S_g)$ induces an isomorphic subgroup $H'$ of $\mathrm{Mod}(S_{g_0,1})$ or $\mathrm{Mod}(S_{g_0',1})$, leading to the same contradiction. Note that no element of $H$ can permute the two boundary components of $S_g\setminus\{c\}$ as $H$ can always be generated by two mapping classes of odd order. Therefore, the assertion follows.  In the case of the semidirect product,  the assertion follows from the fact that $A_n \times \Z < A_n \rtimes \Z$.
\end{proof}

\begin{prop}
\label{prop_multicurve}
Let $H\leq \map(S_g)$ such that $H \cong A_n$. If $H$ preserves a multicurve $\mathcal{C}$ then $|\mathcal{C}| \geq n$. Moreover, if $\mathcal{C}$ contains both separating and non-separating curves then $|\mathcal{C}| \geq 2n$.
\end{prop}

\begin{proof}
Suppose $\mathcal{C}=\{c_1,c_2,\dots,c_r\}$ is a multicurve and $H$ preserves $\mathcal{C}$. Then, any $F\in H$ induces a permutation $\Pi_F$ on $r$ simple closed curves, giving rise to a homomorphism from $H$ to $\Sigma_r$ by sending $F \mapsto \Pi_F$. We claim that the kernel of this homomorphism can not be the whole $H$ as $H$ would otherwise preserve each simple closed curve $c_i$ in $\mathcal{C}$, violating the assertion in Lemma~\ref{lem:scc}. Since $H$ is simple, it forces the homomorphism to be injective, which implies $|\mathcal{C}| \geq n$.

Moreover, if $\mathcal{C}$ contains both separating and non-separating curves then $\mathcal{C}$ can be written as $\mathcal{C}=\mathcal{C}_1\cup \mathcal{C}_2,$ where $\mathcal{C}_1$ is the collection of separating curves and $\mathcal{C}_2$ is the collection of non-separating curves. Following the same line of arguments, as above, for each sub-collection individually, we obtain $|\mathcal{C}|=|\mathcal{C}_1|+|\mathcal{C}_2|\geq 2n$.
\end{proof}\

\begin{cor}
Let $H < \map(S_g)$ be such that $H \cong A_n \rtimes \Z$, where the $\Z$-factor is generated by a multitwist associated with a multicurve $\mathcal{C}$ in $S_g$.  Then $|\mathcal{C}| \geq n$.
\end{cor}

\begin{proof}
As we have seen before (in Lemma~\ref{lem:scc}),  it suffices to consider the case of the direct product.  Since the multitwist commutes with every element of the $A_n$-factor of $H$, it follows (from a similar argument as in proof of Lemma~\ref{lem:scc}) that $H$ preserves $\mathcal{C}$. 
Our assertion now follows from Proposition~\ref{prop_multicurve}.
\end{proof}

We conclude this section with the examples of liftable Dehn twist under alternating covers, leading to $A_n \times \Z$ subgroup of $\map(S_g)$.

\begin{exmp}
Consider the alternating cover induced by the icosahedral subgroup $H<\map(S_{19})$, described in Example~\ref{eg:icosa}. Consider a multicurve $\mathcal{C}=\{c_1,c_2,\dots,c_{20}\}$ where $c_i$'s are disjoint, essential simple closed curves corresponding to each face of the icosahedron. Note that $H$ preserves the multicurve $\mathcal{C}$ and acts on it with exactly one orbit. Let $\bar{c}$ denote the image of $c_i$'s inside $S_{19}/ H$. Then, from the discussion above, we conclude that $T_{\bar{c}}$ lifts under the cover: $S_{19} \longrightarrow S_{19}/ H$, resulting in an $A_5 \times \Z$-subgroup of $\map(S_{19})$.

Applying similar argument as above, we obtain $A_n \times \Z$-subgroups of $\map(S_g)$ for the alternating covers induced by the symmetry group of other Platonic solids.

\end{exmp}
\subsection{Classification of the weak conjugacy classes in $\map(S_g)$ for $g=10,11$}
\label{sec:table}
In this section, we will classify the weak conjugacy classes of the alternating and $\E$-subgroups of $\map(S_{g})$ for $g=10,11$, using Theorems~\ref{thm:main1} and \ref{thm:main2}, respectively. The weak conjugacy classes of the alternating and $\E$-subgroups of $\map(S_{10})$ are listed in Table~\ref{table1}. The weak conjugacy classes of the alternating and $\E$-subgroups of $\map(S_{11})$ are listed in Table~\ref{table2}. Our choice of $S_{10}$ and $S_{11}$ is motivated by the fact that these surfaces admit various $\E$-actions up to weak conjugacy.
 
To achieve our classification for $g =10, 11$, we first need to analyze the $\E$ groups that can act on $S_g$. In this regard, we use the signatures for all finite group actions, along with their GAP IDs, on $S_g$ for $2 \leq g \leq 48$ available at \cite{paulhus}. To find out the GAP IDs of groups in the $\E$-family of order up to $84\cdot (11-1)=840$, we run a very naive code in GAP 4.13.1~\cite{GAP}. We match these two sets of data to identify which group in this family acts with which signature on $S_g$ for $g=10,11$. With this information in hand, we construct all possible surface kernel epimorphisms by running another script in Mathematica 13.1.0~\cite{mathematica}. Finally, we choose a representative surface kernel epimorphism from each weak conjugacy class and record the $\E$-data set associated with it in the table.

In addition, in the table, we compute the standard cyclic factors corresponding to a weak conjugacy class of alternating subgroups of $\map(S_g)$ for $g=10,11$ using Proposition~\ref{prop:cyclicgen}. Moreover, using Propositions~\ref{prop:wlp} and \ref{prop:wls}, we compute the weak liftable pair corresponding to a weak conjugacy class of alternating and $\E$-subgroups of $\map(S_g)$ for $g=10,11$. We avoid writing $\Pi_{\overline{G}}$ in a weak-liftable pair (in short, WLP) to make the tables clutter-free.

The following conclusions can be drawn from the data in Tables~\ref{table1}-\ref{table2}.
\begin{enumerate}
\item There are only two weak conjugacy classes of subgroups of $\map(S_{10})$ isomorphic to $A_4$, represented by the following data sets $$\D_a^1=(4,0;[(1 4)(2 3);2]^{[3]},[(1 2 4);3],[(1 3 2);3]^{[2]}), \text{ and } \D_a^2=(4,1;[(1 2)(3 4);2]^{[3]}).$$
\begin{enumerate}
\item For any $H<\map(S_{10})$ with $H \in [\D_a^1]$, the liftable mapping class group $\mathrm{LMod}(\Orb_H)$ is either torsion-free or it contains a periodic mapping class $\overline{G}$ of order $3$ whose cyclic data set $D_{\overline{G}}=(3,0;(1,3),(2,3))$. In the latter case, since $\overline{G}$ has only two fixed points, it will permute the order-$2$ cone points and the order-$3$ cone points of $\D_a^1$, i.e., $\Pi_{\overline{G}}=(1~2~3)(4~5~6)$. Moreover, the weak-liftable pair $(\D_a^1,(D_{\overline{G}},\Pi_{\overline{G}}))$ gives rise to an $A_4 \times \Z_3$-subgroup of $\mathrm{SMod}(S_{10})$.
\item For any $H<\map(S_{10})$ with $H \in [\D_a^2]$, the liftable mapping class group $\mathrm{LMod}(\Orb_H)$ is either torsion-free or it contains a periodic mapping class of order $2,3$ or $6$, whose cyclic data set is one of the following.
\begin{enumerate}
\item $D_{{\overline{G}}_1}=(2,0;(1,2)^{[4]})$. In this case, ${\overline{G}}_1$ can permute the three cone points of $[\D_a^2]$ in two possible ways: it can either fix all three cone points, i.e., $\Pi_{\overline{G}_1}=id$, or it can permute two cone points and fix the other, i.e., $\Pi_{\overline{G}_1}=(1~2)$. However, in both the cases $(\D_a^2,(D_{{\overline{G}}_1},\Pi_{\overline{G}_1}))$ forms a WLS-pair and give rise to two $\Sigma_4$-subgroups of $\mathrm{SMod}(S_{10})$ that are not weakly conjugate. 
\item $D_{{\overline{G}}_2}=(3,1;-)$. Since ${\overline{G}}_2$ is free, $\Pi_{\overline{G}_2}=id$ and $(\D_a^2,(D_{{\overline{G}}_2},\Pi_{\overline{G}_2}))$ forms a weak-liftable pair, giving rise to an $A_4 \times \Z_3$-subgroup of $\mathrm{SMod}(S_{10})$.
\item $D_{{\overline{G}}_3}=(6,0;(1,3),(1,2),(1,6))$. Since ${\overline{G}}_3$ has only one fixed point, it can only permute the three cone points, i.e., $\Pi_{\overline{G}_3}=(1~2~3)$. However, the weak-liftable pair $(\D_a^2,(D_{{\overline{G}}_3},\Pi_{\overline{G}_3}))$ gives rise to an $A_4 \rtimes_{\varphi} \Z_6$-subgroup of $\mathrm{SMod}(S_{10})$.
\end{enumerate}
\end{enumerate}
\item There is only one weak conjugacy class of subgroups of $\map(S_{10})$ isomorphic to $A_5$, represented by:
\[\D_a^3=(5,0;[(1 5)(2 4);2],[(2 4)(3 5);2],[(2 3)(4 5);2],[(1 2 3 4 5);5]).\]
For any $H<\map(S_{10})$ with $H \in [\D_a^3]$, the liftable mapping class group $\mathrm{LMod}(\Orb_H)$ is either torsion-free or it contains a periodic mapping class $\overline{G}$ of order $3$ whose cyclic data set $D_{\overline{G}}=(3,0;(1,3),(2,3))$. In the latter case, since $\overline{G}$ has only two fixed points, it will permute the three order-$2$ cone points and fix the order-$5$ cone point of $\D_a^3$, i.e., $\Pi_{\overline{G}}=(1~2~3)$. Furthermore, the weak-liftable pair $(\D_a^3,(D_{\overline{G}},\Pi_{\overline{G}}))$ gives rise to an $A_5 \times \Z_3$-subgroup of $\mathrm{SMod}(S_{10})$.
\item There is only one weak conjugacy class of subgroups of $\map(S_{10})$ isomorphic to $A_6$, represented by $\D_a^4=(6,0;[(1 2)(4 6);2],[(1 2 4 3)(5 6);4],[(2 3 4 5 6);5])$. For any $H<\map(S_{10})$ with $H \in [\D_a^4]$, the liftable mapping class group $\mathrm{LMod}(\Orb_H)$ is torsion-free. Furthermore, by Corollary~\ref{cor:self_normalizing}, one can actually conclude that $\mathrm{LMod}(\Orb_H)$ is trivial. Thus, such $H<\map(S_{10})$ are self-normalizing alternating groups.
\item There is only one weak conjugacy class of subgroups of $\map(S_{11})$ isomorphic to $A_4$, represented by $\D_a^5=(4,0;[(1 2)(3 4);2]^{[2]},[(1 2 4);3],[(1 3 4);3]^{[2]},[(2 3 4);3])$. For any $H<\map(S_{11})$ with $H \in [\D_a^5]$, the liftable mapping class group $\mathrm{LMod}(\Orb_H)$ is either torsion-free or it contains a periodic mapping class $\overline{G}$ of order $2$ whose cyclic data set $D_{\overline{G}}=(2,0;(1,2),(1,2))$. In the latter case, since $\overline{G}$ has only two fixed points, it can permute the six cone points of $[\D_a^5]$ in three possible ways; namely, $\Pi_{\overline{G}}^1=(3~4)(5~6)$, $\Pi_{\overline{G}}^2=(1~2)(3~4)(5~6)$, and $\Pi_{\overline{G}}^3=(1~2)(3~4)$. Furthermore, $(\D_a^5,(D_{\overline{G}},\Pi_{\overline{G}}^1))$ and $(\D_a^5,(D_{\overline{G}},\Pi_{\overline{G}}^2))$ form two WLS-pairs and give rise to two $\Sigma_4$-subgroups of $\mathrm{SMod}(S_{11})$ that are not weakly conjugate. However, $(\D_a^5,(D_{\overline{G}},\Pi_{\overline{G}}^3))$ forms a weak-liftable pair and gives rise to an $A_4\times\Z_2$-subgroups of $\mathrm{SMod}(S_{11})$.

\item There is only one weak conjugacy class of subgroups of $\map(S_{11})$ isomorphic to $A_5$, represented by $\D_a^6=(5,0;[(1 3)(2 4);2]^{[2]},[(3 5 4);3],[(3 4 5);3])$. Note that this is the dodecahedral action on $S_{11}$, already explained in detail in Example~\ref{exmp:lift_dodeca}.
\end{enumerate}

\begin{landscape}
\begin{table}
\vfill
    \centering
    \renewcommand{\arraystretch}{1.6}
    \resizebox{22cm}{!}{
    \begin{tabular}{||c||c||c||}\hline\hline
\textbf{Groups} & \textbf{Weak conjugacy classes in Mod($S_{10}$)} & \textbf{Cyclic factors $[D_{\sigma_a};D_{\tau_a}]$ or WLP $[\D_a;D_{\overline{G}}]$}\\
\hline
\hline
\multirow{3}{*}{$A_4$} & $\D_a^1=(4,0;[(1 4)(2 3);2]^{[3]},[(1 2 4);3],[(1 3 2);3]^{[2]})$ & $[(3,3;(2,3)^{[3]});~(3,3;(1,3)^{[3]})]$\\
\cline{2-3} & \makecell{$\D_a^2=(4,1;[(1 2)(3 4);2]^{[3]})$} & $[(3,4;-);~(3,4;-)]$\\
\hline
$A_5$ & $\D_a^3=(5,0;[(1 5)(2 4);2],[(2 4)(3 5);2],[(2 3)(4 5);2],[(1 2 3 4 5);5])$ & $[(3,4;-);~(5,2;(1,5),(4,5))]$\\
\hline
$A_6$ & $\D_a^4=(6,0;[(1 2)(4 6);2],[(1 2 4 3)(5 6);4],[(2 3 4 5 6);5])$ & $[(3,4;-);~(5,2;(1,5),(4,5))]$\\
\hline
\multirow{2}{*}{$\Sigma_4$} & $((4,2,1),0;[((1 4)(2 3),\overline{0});2,1],[((2 1 4),\overline{1});2,2],[((1 2)(3 4),\overline{1});2,2]^{[2]},[((1 3 4), \overline{1});4,2])$ & $[\D_a^2;~(2,0;(1,2)^{[4]})]$\\
\cline{2-3} & \makecell{$((4,2,1),0;[((2 1 3),\overline{1});2,2],[((1 4 3),\overline{1});4,2],[((1 3 4),\overline{1});4,2]^{[2]})$} & $[\D_a^2;~(2,0;(1,2)^{[4]})]$ \\
\hline
\multirow{2}{*}{$A_4 \times \mathbb{Z}_3$} & $\bigl((4,3,0),0;[((1 2)(3 4),\overline{0});2,1],[((1 2 4),\overline{0});3,1],[((2 3 4),\overline{1});3,3],[((2 3 4),\overline{2});3,3]\bigr)$ & $[\D_a^1;~(3,0;(1,3),(2,3))]$\\
\cline{2-3} & \makecell{$\bigl((4,3,0),1;[((1 2)(3 4),\overline{0});2,1]\bigr)$} & $[\D_a^2;~(3,1;-)]$\\
\hline
$A_4 \rtimes \mathbb{Z}_6$ & $\bigl((4,6,1),0;[((1 2 3),\overline{2});3,3],[((1 3)(2 4),\overline{3});4,2],[((1 2 3),\overline{1});6,6]\bigr)$ & $[\D_a^2;~(6,0;(1,3),(1,2),(1,6))]$\\
\hline
$A_5 \times \mathbb{Z}_3$ & $\bigl((4,3,0),0;[((1 5)(3 4),\overline{0});2,1],[((2 5 3),\overline{1});3,3],[((1 2 3 4 5),\overline{2});15,3]\bigr)$ & $[\D_a^3;~(3,0;(1,3),(2,3))]$\\
\hline
\hline
    \end{tabular}}
    \caption{The weak conjugacy classes of alternating and $\E$-subgroups of $\map(S_{10})$.}
    \label{table1}
 \end{table}

\begin{table}
\vfill
    \centering
    \renewcommand{\arraystretch}{1.6}
    \resizebox{22cm}{!}{
\begin{tabular}{||c||c||c||}\hline \hline
\textbf{Groups} & \textbf{Weak conjugacy classes in Mod($S_{11}$)} & \textbf{Cyclic factors $[D_{\sigma_a};D_{\tau_a}]$ or WLP $[\D_a;D_{\overline{G}}]$}\\
\hline\hline
$A_4$ & $\D_a^5=(4,0;[(1 2)(3 4);2]^{[2]},[(1 2 4);3],[(1 3 4);3]^{[2]},[(2 3 4);3])$ & $[(3,3;(1,3)^{[2]},(2,3)^{[2]});~(3,3;(1,3)^{[2]},(2,3)^{[2]})]$ \\
\hline
$A_5$ & $\D_a^6=(5,0;[(1 3)(2 4);2]^{[2]},[(3 5 4);3],[(3 4 5);3])$ & $[(3,3;(1,3)^{[2]},(2,3)^{[2]});~(5,3;-)]$  \\
\hline
\multirow{2}{*}{$\Sigma_4$} & $((4,2,1),0;[((1 3 4),\overline{0});3,1],[((2 3 4),\overline{0});3,1],[((1 3 4),\overline{1});4,2]^{[2]})$ & $[\D_a^5;~(2,0;(1,2)^{[2]})]$\\
\cline{2-3} & \makecell{$((4,2,1),0;[(id, \overline{1});2,2],[((1 2 3), \overline{1});2,2],[((1 4)(2 3),\overline{0});2,1],[((1 4 3), \overline{0});3,1]^{[2]})$} & $[\D_a^5;~(2,0;(1,2)^{[2]})]$  \\
\hline
\multirow{2}{*}{$\Sigma_5$} & $((5,2,1),0;[((1 3 2),\overline{0});3,1], [((1 5 4),\overline{1});4,2], [((2 1 3 4 5),\overline{1});4,2])$ & $[\D_a^6;~(2,0;(1,2)^{[2]})]$\\
\cline{2-3} & \makecell{$((5,2,1),0;[((1 4)(2 3),\overline{0});2,1], [((2 3 5 1 4),\overline{1});6,2], [((3 4 5),\overline{1});6,2])$} & $[\D_a^6;~(2,0;(1,2)^{[2]})]$ \\
\hline
$A_4 \times \mathbb{Z}_2$ & $\bigl((4,2,0),0;[((1 2)(3 4),\overline{0});2,1],[((1 2 4),\overline{0});3,1],[((2 3 4),\overline{1});6,2],[((2 3 4),\overline{1});6,2]\bigr)$ & $[\D_a^5;~(2,0;(1,2)^{[2]})]$ \\
\hline
$A_5\times \mathbb{Z}_2$ & $\bigl((5,2,0),0;[((1 4)(2 5),\overline{0});2,1],[((1 4)(3 5),\overline{1});2,2],[((2 3)(4 5),\overline{1});2,2],[((3 4 5),\overline{0});3,1]\bigr)$ & $[\D_a^6;~(2,0;(1,2)^{[2]})]$ \\
\hline\hline
    \end{tabular}}
    \caption{The weak conjugacy classes of alternating and $\E$-subgroups of $\map(S_{11})$.}
    \label{table2}
\end{table}
\end{landscape}
\section*{Acknowledgements}The authors would like to thank Ravi Tomar  and Pankaj Kapari for some helpful discussions. The second author is supported by an ANRF, Department of Science and Technology, India, File No. MTR/2023/000610.

\bibliographystyle{plain}
\bibliography{alt_lift}
\end{document}